\documentclass[a4paper,10pt]{article}

\usepackage[pagewise]{lineno}%\linenumbers
\usepackage{amsfonts,amssymb,amsmath}

\usepackage[labelsep=endash]{caption}

\usepackage[bookmarks=true,
             bookmarksnumbered=true, breaklinks=true,
             pdfstartview=FitH, hyperfigures=false,
             plainpages=false, naturalnames=true,
             colorlinks=true,pagebackref=false,
             pdfpagelabels]{hyperref}

\usepackage{color}
\usepackage[normalem]{ulem}
\usepackage{tikz}
\usepackage{tikz-cd}
\input{xy}
\usepackage[all]{xy}
\xyoption{all}
\xyoption{poly}

\newtheorem{theo}{Theorem}[section]
\newtheorem{coro}[theo]{Corollary}
\newtheorem{lemm}[theo]{Lemma}
\newtheorem{prop}[theo]{Proposition}
\newtheorem{defi}[theo]{Definition}
\newtheorem{rema}[theo]{Remark}
\newtheorem{exam}[theo]{Example}

\newenvironment{proof}{\noindent \textbf{{Proof.}} \sf}

\def\qed{\hfill $\diamond$ \bigskip}

\newcommand\CC{{\mathbb{C}}}

\def\C{{\mathcal C}}

\def\lim{\mathop{\rm lim}\nolimits}

\def\End{\mathop{\sf End}\nolimits}

\def\JJ{\mathsf J}

\def\CC{\mathsf C}

\def\DD{\mathsf D}
\def\HH{\mathsf H}
\def\AA{\mathbb A}
\def\MM{\mathbb M}

\def\KK{\mathsf K}
\def\ZZ{\mathsf Z}
\def\dd{\mathsf d}
\def\Ext{\mathsf{Ext}}
\def\Hom{\mathsf{Hom}}
\def\Tor{\mathsf{Tor}}

\def\Ker{\mathsf{Ker}}

\def\Coker{\mathsf{Coker}}
\def\coker{\mathop{\rm Coker}\nolimits}

\begin{document}
\sf

\title{Hochschild cohomology of algebras arising from categories and from bounded quivers}
\author{Claude Cibils,  Marcelo Lanzilotta, Eduardo N. Marcos,\\and Andrea Solotar
\thanks{\footnotesize This work has been supported by the projects  UBACYT 20020130100533BA, PIP-CONICET 11220150100483CO, USP-COFECUB, and MATHAMSUD-REPHOMOL.
The third mentioned author was supported by the thematic project of FAPESP 2014/09310-5. The fourth mentioned author is a research member of CONICET (Argentina).}}

\date{}
\maketitle
\begin{abstract}
The main objective of this paper is to provide a theory for computing the Hochschild cohomology of algebras arising from a linear category with finitely many objects and zero compositions. For this purpose, we consider such a category using an ad hoc quiver $Q$,  with an algebra associated to each vertex and a bimodule to each arrow. The computation relies on cohomological functors that we introduce, and on the combinatorics of the quiver. One point extensions are occurrences of this situation, and Happel's long exact sequence is a particular case of the long exact sequence of cohomology that we obtain via the study of trajectories of the quiver. We introduce cohomology along paths, and we compute it under suitable Tor vanishing hypotheses. The cup product on Hochschild cohomology enables us to describe the connecting homomorphism of the long exact sequence.

Algebras arising from a linear category where the quiver is the round trip one, provide square matrix algebras which have two algebras on the diagonal and two bimodules on the corners.
If the bimodules are projective, we show that five-terms exact sequences arise. If the bimodules are free of rank one, we provide a complete computation of the Hochschild cohomology. On the other hand, if the corner bimodules are projective without producing new cycles, Hochschild cohomology in large enough degrees is that of the product of the algebras on the diagonal.

As a by-product, we obtain some families of bound quiver algebras which are of infinite global dimension, and have
Hochschild cohomology zero in large enough degrees.

\end{abstract}



\normalsize

\section{\sf Introduction}

Hochschild cohomology of an algebra over a field is an interesting and not fully understood tool, see for instance \cite{AVRAMOVBUCHWEITZERDMANNLODAYWHITHERSPPON}. It has been defined in 1945 by Hochschild in \cite{HOCHSCHILD1945}, it provides the theory of infinitesimal deformations and the deformation theory on the variety of algebras of a fixed dimension, see \cite{GERSTENHABER1964}.
Moreover it is a Gerstenhaber algebra, and it is related with the representation theory of the given algebra, see for example \cite{ASSEMREDONDOSCHIFFLER,BUSTAMANTE}. Rephrasing the introduction of \cite{KOENIGNAGASE}, observe that Hochschild cohomology is not functorial, there is no natural way to relate the Hochschild cohomology of an algebra to that of its quotient algebras or of its subalgebras. In exchange, the idea is to find a way of relating the cohomology of an algebra to that of an easier or smaller algebra. Several articles go in this direction, starting with the work of Happel in \cite{HAPPEL}, see also \cite{BELIGIANNISJPAA, BUCHW, CMRS2002, GREENMARCOSSNASHALL, GREENSOLBERG, HERMANN, KOENIGNAGASE, MICHELENAPLATZEC}.

In this paper we mainly consider associative algebras over a field $k$, not necessarily finite dimensional, which arise from a $k$-category with finitely many objects and zero compositions.   These algebras are cleft singular extensions $\Lambda = A\oplus M$ - see  \cite[p. 284]{MACLANE},  in addition they are provided with a complete set $E\subset A$ of orthogonal idempotents which are not necessarily primitive and $M$ is an $A$-bimodule verifying $M^2=0$. This setting provides an ad hoc quiver as follows: the vertices are the idempotents of $E$, and there is an arrow from a vertex $x$ to a different one $y$ in case $yMx$ is not zero. This quiver coincides with the Peirce $E$-quiver of the algebra $\Lambda$, see Definition \ref{Peircequiver}.

We develop tools to compute the Hochschild cohomology of $\Lambda$, in relation to cohomological functors that we introduce. They are related to this specific quiver, to the algebras and to the bimodules involved.

The structure that we sketch below is equivalent to a $k$-category, nevertheless our work is based on the combinatorics and the information which it conveys. Recall that a quiver $Q$ is simply laced (see for instance \cite{KIRICHENKOPLA}) if it has neither double arrows nor loops. For a simply laced quiver $Q$, we consider a $Q$-set $\Delta$, which consists on a collection of algebras $\{A_x\}$ associated to each vertex $x$, a set of bimodules $\{M_a\}$ associated to each arrow $a$, and a family of bimodule maps $\alpha$ verifying  associativity constraints. To such a $Q$-set, we associate a $k$-algebra $\Lambda_\Delta$, which is in fact isomorphic to the $k$-algebra arising from the $k$-category determined by the $Q$-set.  If the bimodule maps of $\alpha$ are zero, then $\Lambda_\Delta$ is called a $k$-algebra arising from a $k$-category with zero compositions. In this case, $\Lambda_\Delta$ is a cleft singular extension as before.

A one point extension is an occurrence of a $k$-algebra arising from a $k$-category with zero compositions, it is built on the simply laced quiver which has just an arrow.  In Section \ref{section arising} we  make precise the definitions and we provide other examples of algebras arising from categories with zero compositions.

In Section \ref{HHarising} we consider a $k$-algebra arising from a $k$-category, built on a simply laced quiver $Q$ and a $Q$-set. We analyze  the complex of cochains computing the Hochschild cohomology of the algebra, relative to the subalgebra given by the vertices of $Q$. The main tool that we introduce are the trajectories over $Q$: a trajectory is an oriented path of $Q$ provided with non negative integers at each vertex, that we call waiting times. The duration of the trajectory is the sum of the waiting times plus the length of the path. The cochains of the complex decompose along the trajectories, and we describe the coboundary with respect to this decomposition.

In Section \ref{alongexactseq}, we focalize on algebras arising from $k$-categories with zero compositions. Each non cycle $\delta$ of $Q$ provides a subcomplex of the previous complex given in each degree by the trajectories over $\delta$ whose duration equals the degree. We obtain this way a short exact sequence of complexes, such that  the quotient complex decomposes as a direct sum along cycles of $Q$; moreover this decomposition is based on trajectories over cycles of $Q$.
The above sketched analysis shows that for any path $\omega$ of $Q$, it is natural to define a cohomology theory $\HH^\bullet_\omega(\Delta)$ along $\omega$, based on the $Q$-set $\Delta$. We infer a cohomology long exact sequence from the short exact sequence, which makes use of the cohomology theory along paths of $Q$.

Section \ref{HHalongpaths} mostly concerns the computation of the cohomology along paths of a $Q$-set. If the path is a vertex $x$, trajectories over it are just waiting times on $x$, and the cohomology along $x$ is the Hochschild cohomology of the algebra $A_x$. A main result is that the cohomology along an arrow $a$ is isomorphic to the self extension of $M_a$ with itself, with a shift of one in the degrees. This already shows that the long exact sequence that we  obtain coincides with the long exact sequence of Happel for one point extensions \cite{HAPPEL}, as well as with its generalization for corner algebras obtained independently in \cite{CIBILS2000}, \cite{MICHELENAPLATZEC} and \cite{GREENSOLBERG}. In order to go further in the computation of the cohomology  along paths, Tor vanishing hypotheses are needed. More precisely, let $\omega$ be a path of length two: if Tor between the bimodules of the $Q$-set is zero in positive degrees, then $\HH^\bullet_\omega(\Delta)$ is an Ext functor, shifted by two in the degree. We provide a generalization of this result for paths of higher length. Observe that the Tor vanishing conditions that we require in order to compute the cohomology along paths resemble the ones required for an ideal to be stratified \cite{CLINEPARSHALLSCOTT,KOENIGNAGASE}, as well as the hypotheses used recently in \cite{HERMANN}.

In Section \ref{multiplicative} we describe the connecting homomorphism $\nabla$ of the long exact sequence using the multiplicative structure which is involved. We show that cohomology along paths has a cup product which is compatible with composition of paths. There is a canonical element in the cohomology along paths, that is the sum of the identity maps as endomorphisms of each bimodule, which provides a $1$-cocycle in the sum of the cohomologies along arrows. This enables to describe $\nabla$ as the graded commutator with this canonical element - this way we recover a result in \cite{CMRS2002} as a particular case. In this process we reobtain part of the results of \cite{GREENMARCOSSNASHALL}, for instance that for one point extensions the connecting morphism is a graded algebra map.

In Section \ref{square} we specialize our results to square algebras, namely algebras built on the round trip quiver $Q=\cdot\rightleftarrows\cdot$  provided with a $Q$-set. If the  bimodules associated to the arrows - that is the corner bimodules of the $2\times 2$ matrix algebra - are projective, and if the bimodule maps are zero, the algebra is called a null-square projective algebra. For these algebras we show that the cohomology long exact sequence  splits into five-term exact sequences. Using this fact, we provide explicit formulas for the Hochschild cohomology of a null-square projective algebra $\Lambda$, in terms of the Hochschild cohomology of the algebras at the vertices of $Q$ - that is the algebras on the diagonal - and the kernel and cokernel in even degrees of  the non trivial part $\nabla'$ of $\nabla$. As a consequence, the Hochschild cohomology of the algebras of the $Q$-set is included in the Hochschild cohomology of $\Lambda$.  Moreover the inclusion is canonical  in even degrees while  in odd degrees it is obtained by choosing a splitting of a canonical short exact sequence.

Next we prove in Section \ref{square} that a square algebra which corner bimodules are free of rank one is necessarily a null-square algebra. In other words, the family of bimodule maps is zero for a $Q$-set on the round trip quiver $Q$,  if the bimodules associated to the arrows are free of rank one. We  show that for such a $k$-algebra, $\nabla'$ is injective in even positive degrees. In the finite dimensional case, this leads to explicit formulas for the dimension of the Hochschild cohomology. We also describe a $k$-algebra of this sort as a bound quiver algebra, whenever the algebras on the diagonal are provided as bound quiver algebras.

In Section \ref{Peirce} we consider once more null-square algebras, but focussing on an opposite family to the one studied in Section  \ref{square}. We consider corner projective bimodules provided by some pairs of idempotents belonging to complete systems of orthogonal idempotents of the algebras which are on the diagonal. This leads to a combinatorial data encoded in the Peirce quiver, which conveys enough information for Hochschild cohomology computations.  We prove that if the corner projective bimodules do not produce new oriented cycles in the Peirce quiver then, in large enough degrees, the Hochschild cohomology of the null-square algebra coincides with that of the diagonal algebra. This is different from the results of Section  \ref{square}, indeed the free rank one corner bimodules produce plenty of new cycles in the Peirce quiver.

In Section \ref{examples} we specialize our results to consider examples such as toupie algebras. We also provide
families of
 finite dimensional algebras of infinite global dimension which Hochschild cohomology vector spaces are zero
for large enough degrees. The first example of such a $k$-algebra is provided
 in \cite{BUCHWGREENMADSENSOLBERG}, which gives an answer to a remark by D. Happel in \cite{HAPPEL}.

To end this Introduction, we point out a possible follow-up to our work. In \cite{SNASHALLSOLBERG} the authors introduced a theory of support varieties for modules over any Artin algebra $A$. They consider the Hochschild cohomology ring divided by the ideal generated by the homogeneous nilpotent elements. Observe that the odd degree elements have square zero, so they are nilpotent. This ring is commutative. Whenever it is finitely generated, its spectrum  defines  a variety. The support variety of a module over it is defined. We believe that our results can be used to compare the associated rings  of the algebras involved in our work.
For instance  Theorem \ref{imagem} shows that for a $k$-algebra arising from a $k$-category, the image of the considered map
is included in the ideal generated by the nilpotent elements of the Hochschild cohomology. Also, from Corollary \ref{prodofcoho} we  infer that for $\Lambda=\left(
                                                                                                          \begin{array}{cc}
                                                                                                            A & N \\
                                                                                                            M & B \\
                                                                                                          \end{array}
                                                                                                        \right)$
a null-square projective algebra, with  $\AA=A\times B$ and $\MM=M\oplus N$,  if  there is a positive integer $h$ such that $\MM^{\otimes_\AA h}=0$,
then the previous commutative rings associated to $\Lambda$ and to $\AA$ are isomorphic in large enough degrees.

\section{\sf Algebras arising from $k$-categories with finitely many objects}\label{section arising}

A \emph{$k$-category} is a small category $\C$ enhanced over the category of $k$-vector spaces. In more detail, let $\C_0$ be the set of objects of $\C$; the set of morphisms from $x\in \C_0$ to $y\in \C_0$ is a $k$-vector space  denoted by $_y\C_x$. The composition of $\C$ is  $k$-bilinear, and for all $x\in \C_0$ the image of the canonical inclusion $k \hookrightarrow  {}_x\C_x$ is central in $_x\C_x$. The vector space ${}_x\C_x$ is endorsed with a $k$-algebra structure, while for $x,y\in \C_0$ the vector space $_y\C_x$ is a ${}_y\C_y - {}_x\C_x$ bimodule. In \cite{MITCHELL} a $k$-category is called an ``algebra with several objects".

\begin{defi}
Let $\C$ be a $k$-category with finitely many objects. The \emph{$k$-algebra arising from $\C$} is $a(\C) = \displaystyle{\oplus_{x,y \in \C_0}\  _y\C_x}$, with product given by the composition of $\C$ combined with the matrix multiplication.
\end{defi}

We observe that the unit of $a(\C)$ is the sum of the identities at each object. Moreover $A= \displaystyle{\oplus_{x \in \C_0}\  _x\C_x}$ is a subalgebra, and $M=\displaystyle{\oplus_{x\neq y \in \C_0}\  _y\C_x}$ is an $A$-bimodule.

\begin{defi}
A $k$-category $\C$ has \emph{zero compositions} if for any three different objects $x$, $y$ and $z$ we have $_z\C_y\  _y\C_x=0$.
\end{defi}

Observe in this case $M$ is a two-sided ideal of $a(\C)$.

\begin{defi}
A \emph{system} $E$ of a $k$-algebra $\Lambda$ is a finite set of orthogonal idempotents of $\Lambda$ which is \emph{complete}, namely $\sum_{x\in E}x=1$.
\end{defi}

Note  that for a $k$-category $\C$, the set of identities $E=\{1_x \mid x\in\C_0\}$ is a system of the $k$-algebra $a(\C)$.

\begin{lemm}\label{oneone}
There is a one-to-one correspondence between $k$-categories with finitely many objects and $k$-algebras equipped with a system.
\end{lemm}

\begin{proof}
To a $k$-category $\C$, we associate its algebra $a(\C)$ with the above mentioned system. Conversely, to a $k$-algebra $\Lambda$ with a system $E$, we associate the $k$-category which set of objects is $E$ and which vector space of morphisms from $x$ to $y$ is $y\Lambda x$. The composition of the $k$-category is given by the product of $\Lambda$.\qed
\end{proof}

We recall that a finite \emph{quiver} $Q$ is a finite set $Q_0$ of vertices, a finite set $Q_1$ of arrows, and two maps $s$ and $t$ from $Q_1$ to $Q_0$ called source and target. In this paper we will only deal with finite quivers. It will be useful to consider the following quiver of a $k$-algebra equipped with a system.

\begin{defi}\label{Peircequiver}
Let $\Lambda$ be a $k$-algebra with a system $E$. The \emph{Peirce $E$-quiver} of $\Lambda$ has set of vertices $E$. For $x\neq y\in E$, there is an arrow from $x$ to $y$ if  $y\Lambda x\neq 0$.
\end{defi}

\begin{rema}
The Peirce $E$-quiver of a $k$-algebra is \emph{simply laced}, meaning that there is at most one arrow from a vertex to another one, and that it contains no loops (see for instance \cite[p. 112]{KIRICHENKOPLA}).
\end{rema}

\begin{defi}
Let $\C$ be a $k$-category with finitely many objects. The \emph{Peirce quiver of $\C$} has set of vertices $\C_0$. For different objects $x$ and $y$, there is an arrow from $x$ to $y$ if $_y\C_x\neq 0$.
\end{defi}

Observe that the Peirce quiver of a $k$-category with finitely many objects is also simply laced.  Moreover, Peirce quivers remain unchanged under the  one-to-one correspondence of Lemma \ref{oneone}.

 The next definition provides ultimately the same information as a $k$-category with finitely many objects, or than as a $k$-algebra equipped with a system. Nevertheless the combinatorics behind will be quite useful in the sequel.

\begin{defi}
 Let $Q$ be a simply laced quiver. A \emph{$Q$-set} $\Delta$  is a set which is the union of the following:

\begin{itemize}

\item a set of algebras $\{A_x\}_{x\in Q_0}$,

\item a set of non zero bimodules $\{{}_{t(a)}M_{s(a)}\}_{ a\in Q_1}$, where  ${}_{t(a)}M_{s(a)}$ is an $A_{t(a)}-A_{s(a)}$-bimodule. We will denote this bimodule also by $M_a$. We agree that ${}_yM_x=0$ if there is no arrow from $x$ to $y$. We also write sometimes ${}_xM_x$ instead of $A_x$.

\item a set $\alpha$ of maps as follows: for vertices $x,y,z$ such that $x\neq y$ and $y\neq z$,  an $A_z-A_x$-bimodule map $$\alpha_{z,y,x}: {}_zM_y\otimes_{A_y} {}_yM_x\to {}_zM_x$$
    such that the natural associativity constraints hold,
 insuring that $\C_\Delta$ defined below is  a category.

\end{itemize}
\end{defi}

To a $k$-category with finitely many objects $\C$ with Peirce quiver $Q$, we associate the evident corresponding $Q$-set. Conversely, we set the following:

\begin{defi}
Let $Q$ be a simply laced quiver and let $\Delta$ be a $Q$-set.
\begin{itemize}

\item The \emph{$k$-category $\C_\Delta$} has set of objects $Q_0$, while $_y\left(\C_\Delta\right)_x = {}_yM_x$ and composition of $\C_\Delta$ is determined by the family $\alpha$. Observe that the Peirce quiver of $\C_\Delta$ is $Q$.

    \item The \emph{$k$-algebra $\Lambda_\Delta$} is defined as the algebra arising from $\C_\Delta$.

    In more detail, $\Lambda_\Delta = A\oplus M$, where
    $A=\times_{x\in Q_0} A_x$ and  $M=\oplus_{a\in Q_1} M_a$. For $a\in Q_1$, note that $M_a$ is an $A$-bimodule by extending its original actions by zero. The product in $M$ is determined by the family $\alpha$.

    \item The $k$-algebra $\Lambda_\Delta$ is  \emph{a $k$-algebra arising from a $k$-category with zero compositions} if the maps of the family $\alpha$ are zero.
\end{itemize}
\end{defi}

\begin{exam}

Recall that if $A_y$ is a $k$-algebra and $M_a$ is a left $A_y$-module, then the one point extension $A_y[M_a]$ is the algebra $\left(
                                                                                                       \begin{array}{cc}
                                                                                                         k & 0 \\
                                                                                                         M_a & A_y \\
                                                                                                       \end{array}
                                                                                                     \right)= \left(\begin{array}{cc}
                                                                                                         A_y & M_a \\
                                                                                                         0 & k \\
                                                                                                       \end{array}
                                                                                                     \right)$.
This is an instance of a $k$-algebra arising from a $k$-category with zero compositions, where $Q=x\cdot \stackrel{a}{\longrightarrow} \cdot y $.
\end{exam}
\begin{exam}
More generally, let $A_x$ and $A_y$ be algebras, let $M_a$ be an $A_y - A_x$-bimodule and let $M_b$ be an $A_x-A_y$-bimodule. The null-square algebra (see \cite{CIBILSREDONDOSOLOTAR2017} or \cite{BUCHW}) is $\left(
                                                                                                          \begin{array}{cc}
                                                                                                            A_x & M_b \\
                                                                                                            M_a & A_y \\
                                                                                                          \end{array}
                                                                                                        \right)$
with matrix multiplication given by the bimodule structures of $M_a$ and $M_b$, and setting $m_am_b=0=m_bm_a$ for all $m_a\in M_a$ and $m_b\in M_b$. This is also an instance of a $k$-algebra arising from a $k$-category with zero compositions. In this case the quiver $Q$ is ${}_x\cdot \rightleftarrows \cdot_{y} $.

\end{exam}

Next we provide a link with bound quiver algebras.
Let $Q$ be a  quiver and let $kQ$ be the path algebra. Let $\langle Q_1\rangle$ be the two-sided ideal of $kQ$ generated by the arrows. A two-sided ideal $I$ of $kQ$ is admissible if $I\subset \langle Q_1 \rangle^2$ and if there exists a positive integer $n$ such that $\langle Q_1\rangle^n\subset I$. The quiver $Q$ is equal to the Gabriel quiver - also called Ext-quiver - of $\Lambda=kQ/I$. Its vertices are the isomorphism classes of simple $\Lambda$-modules; let $S$ and $S'$ be simple $\Lambda$-modules, the number of arrows from $S$ to $S'$ is $\dim_k\Ext^1_\Lambda(S',S)$.

\begin{rema} Let $\Lambda=kQ/I$ be a $k$-algebra as above. The Peirce $Q_0$-quiver of $\Lambda$ is obtained from $Q$ as follows:
\begin{enumerate}
\item
delete all loops,
\item
for $x\neq y$, add an arrow from $x$ to $y$ if there exists a path from $x$ to $y$ in $Q$ which is not in $I$ - that is a path which is not zero in $\Lambda$,
\item
replace any set of arrows sharing the same source $x$, and sharing the same target $y$, by a single arrow from $x$ to $y$.
\end{enumerate}
\end{rema}

Recall that a \emph{cleft singular extension algebra} (see \cite[p. 284]{MACLANE}) is  a $k$-algebra $\Lambda$ with a decomposition $\Lambda = A \oplus M,$ where $A$ is a subalgebra and $M$ is a two-sided ideal of $\Lambda$ verifying  $M^2=0.$ These algebras are also called trivial extensions, see for instance \cite{ASSEMREDONDOSCHIFFLER} and \cite{BELIGIANNISCA, BELIGIANNISJPAA} where the natural generalization for abelian categories is considered.

A $k$-algebra arising from a $k$-category with zero compositions is thus a cleft singular extension $A\oplus M$, which is moreover equipped with a system.

\begin{exam}
Let $R$ be a quiver given by two quivers  $Q_1$ (which we view horizontally upstairs) and $Q_2$  (horizontally downstairs),  and  a finite set of  \emph{vertical down} arrows from some vertices of $Q_1$ to vertices of $Q_2$, as well as a set of \emph{vertical up} arrows. Let $I$ be the two-sided ideal of $kR$ generated by the paths which contain two vertical arrows.

The algebra $\Lambda= kR/I$ arises from a $k$-category with zero compositions. Indeed, for $i=1,2$, let $x_i$ be the sum of  the vertices of $Q_i$. The set $E=\{x_1, x_2\}$ is a system of $\Lambda$. The Peirce $E$-quiver is $Q=\ {}_{x_1}\cdot \rightleftarrows \cdot_{x_2} $, where the arrow from $x_1$ to $x_2$ is called $a$ and the reverse one is called $b$. The algebras of the $Q$-set are $A_{x_1} =kQ_1$ and $A_{x_2} =kQ_2$. Notice that the bimodules $M_a$ and $M_b$ are projective bimodules.
\end{exam}

\section{\sf Hochschild cohomology of algebras arising from categories}\label{HHarising}

In this section we provide tools for computing the Hochschild cohomology of a $k$-algebra arising from a $k$-category with finitely many objects.

Let $\Lambda$ be a $k$-algebra and let $Z$ be a $\Lambda$-bimodule. By definition, the Hochschild cohomology of $\Lambda$ with coefficients in $Z$ is
 $$\HH ^n(\Lambda, Z) = \Ext^n_{\Lambda\otimes \Lambda^{\mathsf{op}}}(\Lambda, Z).$$
For $Z=\Lambda$ it is usual to write $\HH\HH^n(\Lambda)$ instead of $\HH^n(\Lambda, \Lambda).$

The following result is well-known, the proof is analogous to the one sketched in \cite{CIBILSREDONDOSOLOTAR2017} for Hochschild homology:

 \begin{lemm}\label{byseparable}
Let $\Lambda$ be a $k$-algebra, let $D$ be a separable subalgebra of $\Lambda$ and let $Z$ be a $\Lambda$-bimodule. The cohomology
of the  complex $\JJ^{\bullet}(Z)$
\small
$$0\to Z^D\stackrel{d}{\to}\Hom_{D-D}(\Lambda,Z)      \stackrel{d}{\to}   \Hom_{D-D}(\Lambda\otimes_D\Lambda,Z)\cdots       \stackrel{d}{\to} \Hom_{D-D}\left(\Lambda^{\otimes_D n},Z\right)\stackrel{d}{\to} \cdots$$
\normalsize
is $\HH^*(\Lambda, Z)$,
where $Z^D=\{z\in Z \mid dz=zd \mbox{ for all  } d\in D\}$ and $\Hom_{D-D}$ stands for $\Hom_{{D\otimes D^{\mathsf{op}}}}$. The definition of the maps $d$ is provided by the same formulas  that those for computing Hochschild cohomology:

\begin{itemize}
\item
for $n>0$
\begin{align*}
(d f)(x_1\otimes x_2\otimes \dots\otimes x_{n+1}) &= x_1f(x_2\otimes x_3\otimes  \dots \otimes x_{n+1}) \\
&+\sum_{i=1}^{n}(-1)^i f(x_1\otimes \dots\otimes x_ix_{i+1}\otimes \dots\otimes x_{n+1}) \\
&+(-1)^{n+1}f(x_1\otimes x_2\otimes \dots \otimes x_{n})x_{n+1},
\end{align*}
\item
for $n=0$, $z\in Z^D$, and $x\in \Lambda$ we have $(d z)(x)= xz-zx$.
\end{itemize}
\end{lemm}

Let $Q$ be a simply laced quiver with a $Q$-set $\Delta$, and let $\Lambda_\Delta$ be the corresponding algebra, see Section \ref{section arising}. Let $D=\times_{x\in Q_0}k$ be the separable subalgebra of $A=\times_{x\in Q_0} A_x$ given by the inclusions $k\subset A_x$. Next we provide a canonical decomposition of the space of $n$-cochains $\JJ^n(\Lambda_\Delta)$.

For any quiver $Q$ and $m>0$, a \emph{path of length $m$} is a sequence of arrows $\omega = a_m\dots a_1$ which are \emph{concatenated}, that is
$t(a_i)=s(a_{i+1})$ for all $i$. The set of vertices $Q_0$ is the set of paths of length $0$; if $x\in Q_0$, then $s(x)=t(x)=x$. The set of  paths of length $m$ is denoted $Q_m$.
The set of paths of length less or equal to $m$ is denoted $Q_{\leq m}$.
The maps $s$ and $t$ are extended to the set of paths by $s(\omega)=s(a_1)$ and $t(\omega)=t(a_m)$.

A path $\omega$ is a \emph{cycle} if $s(\omega)= t(\omega)$. We denote $CQ_m$ the set of  \emph{cycles} of length $m$, note that  $CQ_0=Q_0$. Its complement in $Q_m$ is the set of \emph{non cycles} of length $m$ that we denote $DQ_m$.

\normalsize

\begin{defi}
Let $Q$ be a simply laced quiver, let $m>0$ and let $\omega=a_m\cdots a_1\in Q_m$.
The set $T_n(\omega)$ of trajectories of duration $n$ over $\omega$ is the set of sequences
$$\tau= t(a_m)^{p_{m+1}},a_m,s(a_m)^{p_m},\cdots,s(a_2)^{p_2},a_1,s(a_1)^{p_1}$$
where each $p_i$ is a non negative integer and $n=m+\sum_1^{m+1}p_i$. The integer $n-m$ is the \emph{total waiting time} of the trajectory. For $p> 0$ and $x\in Q_0$, the symbol $x^p$  denotes the sequence $(x,x,\dots, x)$ where $x$ is repeated $p$ times.
The \emph{waiting time} of the trajectory $x^p$ over $x$ is $p$. If $p=0$, the symbol $x^0$ is the empty sequence,  it corresponds to a $0$ waiting time at $x$.

\end{defi}
We record the following facts:

\begin{itemize}
\item
If $x$ is  a vertex, then $T_n(x)=\{x^n\}$ for $n\geq 0$.
\item
If $\omega\in Q_m$ and $n<m$, then $T_n(\omega)=\emptyset$.
\item
If $\omega\in Q_m$, then $T_m(\omega)$  has a
 unique element  $$t(a_m)^0,a_m,s(a_m)^0,\dots,s(a_2)^0,a_1,s(a_1)^0$$ of total waiting time zero and duration $m$.
\end{itemize}

\begin{rema}
By definition, a trajectory $\tau$ over a path $\omega$ is a sequence of vertices and arrows which are concatenated, that is the product of two successive entries of $\tau$ is not zero in the path algebra $kQ$. Moreover, the  product of all successive entries of $\tau$ is  equal to $\omega$ in $kQ$.
\end{rema}

\begin{defi}
Let $Q$ be a simply laced quiver with a $Q$-set $\Delta$ and let $\omega=a_m\dots a_1\in Q_m$ for $m>0$. For $n\geq m$, let $$\tau= t(a_m)^{p_{m+1}},a_m,s(a_m)^{p_m},\cdots,s(a_2)^{p_2},a_1,s(a_1)^{p_1} \in T_n(\omega).$$
The \emph{evaluation of $\tau$ at $\Delta$}  is the vector space
$$\tau_\Delta= A_{t(a_m)}^{\otimes p_{m+1}}\otimes M_{a_m} \otimes A_{s(a_m)}^{\otimes p_{m}} \otimes\cdots\otimes A_{s(a_2)}^{\otimes p_{2}}\otimes M_{a_1}\otimes A_{s(a_1)}^{\otimes p_{1}}$$
where all the tensor products are over $k$. If $x\in Q_0$ and $\tau\in T_n(x)$, then $\tau_\Delta= A_x^{\otimes n}.$ By definition, if $A$ is a a $k$-algebra, then $A^{\otimes 0}=k$.

\end{defi}

\begin{prop}\label{decompose}
Let $Q$ be a simply laced quiver  with a $Q$-set $\Delta$,  let $\Lambda_\Delta$ be the corresponding algebra, and let $Z$ be a $\Lambda_\Delta$-bimodule. For $n>0$ the following decompositions hold:
\begin{equation}\label{Lambda^n}
(\Lambda_\Delta)^{\otimes_D n} = \bigoplus_{\omega\in Q_{\leq n}}\left[\bigoplus_{\tau \in T_n(\omega)} \tau_\Delta\right]
\end{equation}
\medskip
\begin{equation}\label{HomDLambda^n}
\JJ^n(Z) = \bigoplus_{\omega\in Q_{\leq n}}\left[\bigoplus_{\tau \in T_n(\omega)} \Hom_k\left( \tau_\Delta,\  t(\omega)Z s(\omega)\right) \right].
\end{equation}
Moreover
$$\JJ^0(\Lambda)=\Lambda^D=\times_{x\in Q_0}A_x= \bigoplus_{x\in Q_0}\bigoplus_{\tau\in T_0(x)} \tau_\Delta.$$
\end{prop}

\begin{proof} The proof of (\ref{Lambda^n}) is by induction on $n$, we only describe in detail the low degree cases.
Recall that $\Lambda_\Delta = A\oplus M$, hence
$$(\Lambda_\Delta)\otimes_D (\Lambda_\Delta) = \left(A\otimes_D A\right) \oplus \left(A\otimes_D M\right) \oplus \left(M\otimes_D A\right) \oplus \left(M\otimes_D M\right).$$
For  $x\in Q_0$, let $e_x$ be the idempotent of $D$ with value $1$ at $x$ and $0$ at other vertices. Note that $\{e_x\}_{x\in Q_0}$ is a complete set of central orthogonal idempotents of $A$. Actually $$Ae_x=A_x=e_xA.$$  Observe that if
$x\neq y$, then
$$ae_x\otimes e_ya'=ae_x^2\otimes e_ya'=ae_x\otimes e_xe_ya'=0,$$
so $A_x\otimes_D A_y=0$. Moreover $A_x\otimes_D A_x= A_x\otimes A_x$. Hence
$$A\otimes_D A= \bigoplus_{x\in Q_0} A_x\otimes A_x.$$
The direct summand $A_x\otimes A_x$ corresponds to the trajectory $x^2$ of total waiting time $2$ and duration $2$ at the vertex $x$.

Observe that if $a\in Q_1$, then $M_a= e_{t(a)}Me_{s(a)}=e_{t(a)}(\Lambda_\Delta) e_{s(a)}$, while $e_yMe_x=0$ if there is no arrow  from $x$ to $y$ in $Q$.
$$M=\bigoplus_{a\in Q_1}M_a = \bigoplus_{x,y\in Q_0}e_yMe_x.$$
If $z\neq t(a)$, then $A_z\otimes_D M_a=0$, while $A_{t(a)}\otimes_D M_a= A_{t(a)}\otimes M_a$.
Hence
$$A\otimes_D M = \bigoplus_{a\in Q_1} A_{t(a)}\otimes M_a.$$
Similarly
$$M\otimes_D A = \bigoplus_{a\in Q_1} M_a\otimes A_{s(a)}.$$
Each  direct summand above corresponds to a path of length $1$ - that is an arrow $a$ - and the trajectories $(t(a)^1,a,s(a)^0)$ or $(t(a)^0,a,s(a)^1)$ over $a$, which are of  total waiting time $1$ and duration $2$.
Analogously, we obtain the decomposition
$$M\otimes_D M= \bigoplus_{\{\omega=a_2a_1\in Q_2\}} M_{a_2}\otimes M_{a_1}.$$
Each direct summand corresponds to the unique trajectory over $a_2a_1$ of total waiting time $0$ and duration $2$.

The next observations will prove (\ref{HomDLambda^n}). Recall that $D=\times_{x\in Q_0}ke_x$ is a semisimple algebra.  Hence a $D$-bimodule $U$ has a canonical decomposition into its isotypic components $U=\oplus_{x,y\in Q_0}e_yUe_x$. Observe that each direct summand of (\ref{Lambda^n}) is a $D$-bimodule. More precisely for $\tau\in T_n(\omega)$, we have that $\tau_\Delta$ is a direct summand of the isotypic component $$e_{t(\omega)}\left[(\Lambda_\Delta)^{\otimes_D n}\right]e_{s(\omega)}.$$

Let $U$ and $V$ be $D$-bimodules.
The following is immediate - and it is an instance of Schur's Lemma: if $y\neq t$ or $x\neq u$, then $$\Hom_D(e_yUe_x, e_tVe_u)=0,$$ while $$\Hom_D(e_yUe_x, e_yVe_x)=\Hom_k(e_yUe_x, e_yVe_x).$$
\qed
\end{proof}
Our next aim is to use the decomposition of cochains that we have obtained in Proposition \ref{decompose} in order to describe the coboundary $d$ of Lemma \ref{byseparable}.

\begin{defi}\label{plustrajectories}
Let
$$\tau= t(a_m)^{p_{m+1}},a_m,s(a_m)^{p_m},\cdots,s(a_2)^{p_2},a_1,s(a_1)^{p_1}$$
be a $n$-trajectory over a path $\omega=a_m\dots a_1$. The set $\tau^+$ is the union of:
\begin{itemize}
\item $\tau_0^+$, the set of $n+1$-trajectories obtained by increasing a waiting time of $\tau$ by one.
\item $\tau_1^+$, the set of $n+1$-trajectories
$$t(c)^0, c,s(c)^{p_{m+1}},a_m,s(a_m)^{p_m},\cdots,s(a_2)^{p_2},a_1,s(a_1)^{p_1}$$
where $c\in Q_1$ is any arrow after $\omega$, that is verifying $s(c)=t(a_m)$.

Similarly, $\tau_1^+$ contains also the trajectories obtained by adding any arrow $c$ before $\omega$, that is verifying $t(c)=s(a_1).$

\item
$\tau_2^+$,  the set of $n+1$-trajectories obtained

\begin{itemize}
\item
either by replacing any arrow $a_i$ by some  path $a''_i a'_i$  which is parallel to $a_i$,
\item
or by replacing any waiting time  by some  path $a''_i a'_i$  which is parallel to it.
\end{itemize}

    \end{itemize}
\end{defi}

\begin{defi}\label{counterpart}
Let $\omega$ be a path of $Q$. We denote by $\Delta_\omega$ the $A_{t(\omega)}-A_{s(\omega)}$-bimodule $$t(\omega)(\Lambda_\Delta)s(\omega).$$
\end{defi}

Observe that if there is no arrow parallel to $\omega$ - that is, there is no arrow $a$ such that $s(a)=s(\omega)$ and $t(a)=t(\omega)$ - and  $\omega$ is a non cycle, then $\Delta_\omega = 0$.  On the other hand if such an arrow $a$ exists, then $\Delta_\omega = M_a\neq 0$. If $\omega$ is a cycle, then $\Delta_\omega= A_{s(\omega)}= A_{t(\omega)}$.

\begin{defi}\label{J}
The space of \emph{$n$-cochains} over a trajectory $\tau\in T_n(\omega)$ is
$$\JJ_\tau= \Hom_k\left( \tau_\Delta,\  \Delta_\omega\right).$$
\end{defi}

\begin{prop}\label{preserve}Let $d$ be one of the coboundary maps of Lemma \ref{byseparable}. The following holds
$$d \JJ_\tau \subset \bigoplus_{\sigma \in\tau^+} \JJ_{\sigma}.$$

\end{prop}
\begin{proof}
Let $T_{n}$ be the set of trajectories of duration $n$ and let $\tau\in T_n$. Let $f_\tau \in \JJ_\tau$, write $d f_\tau= \sum_{\sigma \in T_{n+1}} \left(d f_\tau \right)_\sigma$ the decomposition of $d f_\tau$ according to (\ref{HomDLambda^n}). We assert that if $\sigma \notin \tau^+$, then $\left(d f_\tau \right)_\sigma =0.$ Recall that
\begin{align*}
(d f_\tau)(x_1\otimes x_2\otimes \dots\otimes x_{n+1}) &= x_1f_\tau(x_2\otimes x_3\otimes  \dots \otimes x_{n+1}) \\
&+\sum_{i=1}^{n}(-1)^i f_\tau(x_1\otimes \dots\otimes x_ix_{i+1}\otimes \dots\otimes x_{n+1}) \\
&+(-1)^{n+1}f_\tau(x_1\otimes x_2\otimes \dots \otimes x_{n})x_{n+1}.
\end{align*}

Let $x_1\otimes x_2\otimes \dots\otimes x_{n+1}\in \sigma_\Delta$ where $\sigma \notin \tau^+$. We will prove that each summand above is zero.

\begin{itemize}

 \item[-]
 If $x_1f_\tau(x_2\otimes x_3\otimes  \dots \otimes x_{n+1})\neq 0$, then
 $f_\tau(x_2\otimes x_3\otimes  \dots \otimes x_{n+1})\neq 0,$ hence $0\neq x_2\otimes x_3\otimes  \dots \otimes x_{n+1}\in \tau_\Delta$. Moreover, $x_1$ belongs either to $A_{t(\omega)}$ or to a bimodule $M_c$ for $c\in Q_1$ with $s(c)=t(\omega)$.  This is equivalent respectively to $\sigma\in\tau_0^+$ or $\sigma\in\tau_1^+$.
\item[-]
If $f_\tau(x_1\otimes \dots\otimes x_ix_{i+1}\otimes \dots\otimes x_{n+1})\neq 0$, then $0\neq x_1\otimes \dots\otimes x_ix_{i+1}\otimes \dots\otimes x_{n+1}\in \tau_\Delta.$

\begin{itemize}
  \item

If $x_ix_{i+1}$ belongs to a $k$-algebra $A_z$, then

\begin{itemize}
\item
either $x_i\in A_z$ and $x_{i+1}\in A_z$, hence $\sigma$ is obtained from $\tau$ by increasing by one the waiting time at the vertex $z$, that is $\sigma\in\tau_0^+$,
\item
or $x_i\in M_{a''}$ and $x_{i+1}\in M_{a'}$ for a path $a''a'$ which is parallel to the vertex $z$, that is $\sigma\in\tau_2^+$.

\end{itemize}

\item
If $x_ix_{i+1}$ belongs to a bimodule $M_{a} $ for some arrow  $a$ of $\omega$, then
\begin{itemize}
\item
either $x_i\in A_{t(a)}$ and $x_{i+1}\in M_{a}$,  or $x_i\in M_{a}$  and $x_{i+1}\in A_{s(a)}$ that is, $\sigma\in \tau_0^+$
\item

or $x_i\in M_{a''}$ and $x_{i+1}\in M_{a'}$ for a path of  $a''a'$ which is parallel to $a$ that is, $\sigma\in\tau_2^+$.
\end{itemize}

\end{itemize}
\item[-]
If the last summand is non zero, the proof is analogous to the first case.
\end{itemize}
\qed
\end{proof}

\section{\sf  A long exact sequence for algebras arising from categories with zero compositions}\label{alongexactseq}

The result of Proposition \ref{preserve} can be made more precise for algebras arising from $k$-categories with finitely many objects and zero compositions.  Let $Q$ be a simply laced quiver provided with a $Q$-set $\Delta=(A,M)$, where $A=\times_{x\in Q_0}A_x$ and $M=\oplus_{a\in Q_1} M_a$. Let $\Lambda_\Delta$ be the corresponding algebra.

\begin{defi}
Let $\omega\in Q_m$. The vector space of \emph{$n$-cochains along $\omega$} is
$$\JJ^n_\omega= \bigoplus_{\tau\in T_n(\omega)}\JJ_\tau = \bigoplus_{\tau\in T_n(\omega)}\Hom_k\left(\tau_\Delta, \Delta_\omega\right).$$
\end{defi}
\begin{rema}
Let $\omega\in Q_m$. If $n<m$, then $\JJ^n_\omega=0$. If $n=m$ and  $\omega=a_m\dots a_1$, then $$\JJ^m_\omega= \Hom_k(M_{a_m}\otimes\dots\otimes M_{a_1}, \Delta_\omega).$$
\end{rema}

\begin{prop}\label{subcomplex}
If $\delta\in DQ_m$, then $\JJ^\bullet_\delta$ is a subcomplex of $\JJ^{\bullet}(\Lambda_\Delta)$.
\end{prop}

\begin{proof}
We suppose $\Delta_\delta\neq 0$ since otherwise $\JJ_\delta^\bullet=0$.  The path $\delta$ is not a cycle,  so there exists $a\in Q_1$  parallel to $\delta$, and $\Delta_\delta = M_a$.

Let $\delta=a_m\dots a_1$, let $\tau=t(a_m)^{p_{m+1}},a_m,s(a_m)^{p_m},\cdots,s(a_2)^{p_2},a_1,s(a_1)^{p_1}$ be a trajectory of duration $n$ over $\delta$, and let $f_\tau \in \JJ_\tau$.  If $\sigma$ is a trajectory of duration $n+1$ which is not over $\delta$, we will prove that $(d f_\tau)_\sigma=0$.
We  already know from Proposition \ref{preserve} that if $\sigma\notin\tau^+$, then $(d f_\tau)_\sigma =0$.

\begin{itemize}

\item
If $\sigma\in\tau^+_0$, there is nothing to prove since $\sigma$ is over $\delta$, as $\sigma$ is obtained by increasing some waiting time of $\tau$ by one.
\item
If $\sigma\in\tau^+_1$, then $\sigma$ is not over $\delta$ and we will show that $(d f_\tau)_\sigma =0$. Suppose firstly that
$$\sigma= t(c)^0, c,s(c)^{p_{m+1}},a_m,s(a_m)^{p_m},\cdots,s(a_2)^{p_2},a_1,s(a_1)^{p_1}$$
for some arrow $c$ such that $s(c)=t(a_m)$.
   Let $x_1\otimes x_2\otimes \dots\otimes x_{n+1}\in \sigma_\Delta$, we recall that
\begin{align*}
(d f_\tau)(x_1\otimes x_2\otimes \dots\otimes x_{n+1}) &= x_1f_\tau(x_2\otimes x_3\otimes  \dots \otimes x_{n+1}) \\
&+\sum_{i=1}^{n}(-1)^i f_\tau(x_1\otimes \dots\otimes x_ix_{i+1}\otimes \dots\otimes x_{n+1}) \\
&+(-1)^{n+1}f_\tau(x_1\otimes x_2\otimes \dots \otimes x_{n})x_{n+1}.
\end{align*}
We will show that each of the previous summands is zero.
For the first one,  $x_1\in M_c$ and $f_\tau(x_2\otimes x_3\otimes  \dots \otimes x_{n+1})\in M_a$. Hence this  summand belongs to the product $M_c M_a$, which is zero because $\Lambda_\Delta$ is a a $k$-algebra arising from a $k$-category with zero compositions.

We consider now $f_\tau(x_1\otimes \dots\otimes x_ix_{i+1}\otimes \dots\otimes x_{n+1})$ for $i=1,\dots, n$.
 Notice that $x_1\otimes \dots\otimes x_ix_{i+1}\otimes \dots\otimes x_{n+1}\in \tau'_\Delta$ where $\tau'$ is a trajectory over $c\delta$, hence $\tau'\neq\tau$. Then this summand is $0$.

For the last summand,  note that $x_1\otimes x_2\otimes \dots \otimes x_{n}\in \tau'_\Delta$, where $\tau'$ is a trajectory over a path which last arrow is $c$. Note that $c\neq a_m$, since there are no loops in $Q$. Hence $\tau'\neq \tau$ and the summand is $0$.

The case where $\sigma\in \tau_1^+$ is obtained from $\tau$ by adding at the beginning an arrow $c$ such that $t(c)=s(a_1)$ is analogous.
\item
If $\sigma\in\tau_2^+$, suppose
$$\sigma= t(a_m)^{p_{m+1}},a_m,s(a_m)^{p_m},\cdots,a''_j,a_j',\cdots,s(a_2)^{p_2},a_1,s(a_1)^{p_1}$$
and let $x_1\otimes x_2\otimes \dots\otimes x_{n+1}\in \sigma_\Delta.$

The first summand of $(d f_\tau)(x_1\otimes x_2\otimes \dots\otimes x_{n+1})$ is zero since $x_2\otimes \dots\otimes x_{n+1}$ belongs to $\tau'_\Delta$ for $\tau'\neq \tau$ beacuse $\tau'$ is a trajectory over a path different from $\delta$. Then
$$f_\tau(x_2\otimes x_3\otimes  \dots \otimes x_{n+1})=0.$$
The proof that the last summand is zero is analogous.

The other summands are also zero: if $x_{i+1}\in M_{a''_j}$ and $x_i\in M_{a'_j}$, then $x_{i+1}x_i=0$ because the products of elements of the bimodules are zero in $\Lambda_\Delta$. Otherwise, observe that $x_1\otimes \dots\otimes x_ix_{i+1}\otimes \dots\otimes x_{n+1}\in \tau'_\Delta$ where $\tau'$ is a trajectory over a path different from $\delta$, then $$f_\tau(x_1\otimes \dots\otimes x_ix_{i+1}\otimes \dots\otimes x_{n+1})=0.$$
\end{itemize}
\qed
\end{proof}
\begin{defi}\label{deltacoho}
Let $\Lambda_\Delta$ be a $k$-algebra arising from a $k$-category with zero compositions, with respect to a $Q$-set $\Delta=(A,M)$. Let $\delta\in DQ_m$.

For $n\geq m$, the \emph{cohomology of $\Delta$ along $\delta$} is denoted $\HH_\delta^n(\Delta)$, and it is the cohomology of the complex of cochains $\JJ_\delta^\bullet$.
\end{defi}

From the above proof, we record the next result for the coboundary of $\JJ_\delta^\bullet$.

\begin{prop}\label{dftau}
Let $\delta\in DQ_m$ and let $\tau\in T_n(\delta)$. Let $f_\tau\in \JJ_\tau.$ The coboundary $d : \JJ_\delta^n\to \JJ_\delta^{n+1}$ is given by
$$df_\tau = \sum_{\sigma\in \tau^+_0} (df_\tau)_\sigma.$$

\end{prop}
\begin{proof}
We know from Proposition \ref{preserve} that
\begin{equation}\label{d}
df_\tau = \sum_{\sigma\in \tau^+} (df_\tau)_\sigma= \sum_{\sigma\in \tau_0^+} (df_\tau)_\sigma\ +\ \sum_{\sigma\in \tau_1^+} (df_\tau)_\sigma +\ \sum_{\sigma\in \tau_2^+} (df_\tau)_\sigma.
\end{equation}
We have just shown in the previous proof that if $\delta$ is not a cycle, and if $\sigma\in\tau_1^+\cup \tau_2^+$, then $(df_\tau)_\sigma=0$.

 \qed
\end{proof}

\begin{defi}\label{Momega} Let $\omega=a_m \dots a_1\in Q_m$, where $a_i\in Q_1$ for all $i$. The \emph{bimodule along $\omega$} is
$$M_\omega= M_{a_m}\otimes_{A_{s(a_m)}}\cdots \otimes_{A_{s(a_2)}} M_{a_1}.$$
\end{defi}

Note that $M_\omega$ is an $A_{t(\omega)}-A_{s(\omega)}$-bimodule.

\begin{lemm}\label{HHdelta1}
Let $\Delta$ be a $Q$-set, and let $\delta\in DQ_m$. The following holds:
$$\HH_\delta^m(\Delta)= \Hom_{A_{t(\delta)}-A_{s(\delta)}} \left( M_\delta, \Delta_\delta\right).$$
In particular if $a\in Q_1$,
$$\HH_a^1(\Delta)= \End_{A_{t(a)}-A_{s(a)}} M_a.$$
\end{lemm}
\begin{proof}
The complex of cochains $\JJ_\delta^\bullet(\Lambda)$ begins as follows:
\small
$$0\to \Hom_k(M_{a_m}\otimes\dots\otimes M_{a_1},\ \Delta_\delta ) \stackrel{d}{\to} \bigoplus_{\tau\in T_{m+1}(\delta)}\Hom_k(\tau_\Delta, \ \Delta_\delta)\stackrel{d}{\to} \cdots$$
\normalsize
Observe that $T_{m+1}(\delta)=\{\tau_i\}_{i=0,\dots, m+1}$, where $\tau_i$ for $i=0,\dots, m$ is obtained from the trajectory with total waiting time $0$ by increasing the waiting time at $s(a_i)$ by one, while $\tau_{m+1}$ is obtained by increasing by one $t(a_m)$. Hence, for $i=0\dots,m$:
$$(\tau_i)_\Delta=M_{a_m}\otimes\dots\otimes M_{a_i}\otimes A_{s(a_i)}\otimes M_{a_{i-1}}\dots\otimes M_{a_1}$$and
$$(\tau_{m+1})_\Delta = A_{t(a_m)}\otimes M_{a_m}\otimes\dots\otimes M_{a_1}.$$
Let $f\in \Hom_k(M_{a_m}\otimes\dots\otimes M_{a_1},\ \Delta_\delta)$ be such that $d f=0$, that is $(d f)_{\tau_i}=0$ for all $i$.
For $0<i<m+1$,  this shows that $f$ is $A_{s(i)}$-balanced, while the cases $i=0$ and $i=m+1$  prove that $f$ is a morphism of bimodules.\qed

\end{proof}

\begin{rema}
For $r\geq 0$ and $\delta\in DQ_m$,  we will prove in the next section that assuming adequate hypotheses, $$\HH_\delta^{m+r}(\Delta)= \Ext^r_{A_{t(\delta)}-A_{s(\delta)}} \left( M_\delta, \Delta_\delta\right).$$
\end{rema}

We have proven in Proposition \ref{subcomplex} that each non cycle $\delta$ provides a subcomplex $\JJ^\bullet_\delta$ of the complex $\JJ^\bullet(\Lambda_\Delta)$. This enables to consider their direct sum as follows.

\begin{defi}\label{non cyclesubcomplex}
Let $\Lambda_\Delta$ be a a $k$-algebra arising from a $k$-category with zero compositions, determined by a simply laced quiver $Q$ and a $Q$-set $\Delta=(A,M)$. The \emph{non cycle  subcomplex} of cochains $\DD^\bullet(\Lambda_\Delta)$ of $\JJ^\bullet(\Lambda_\Delta)$ is
$$\DD^n(\Lambda_\Delta)=\bigoplus_{\delta\in DQ} \JJ_\delta^n.$$

\end{defi}

\begin{rema}
We already know that if $\omega\in Q_m$ and $n<m$, then  $T_n(\omega)=\emptyset.$ For such $m$ and $n$, $$\DD^n(\Lambda_\Delta)=\bigoplus_{\delta\in DQ_{\leq n}} \JJ_\delta^n.$$ In particular $\DD^0(\Lambda_\Delta)=0$.
Note also that $$\HH^n(\DD^\bullet(\Lambda_\Delta))=\bigoplus_{\delta\in DQ_{\leq n}}\HH^n_\delta (\Delta).$$
\end{rema}

 We consider now the exact sequence of complexes of cochains, where $\Lambda_\Delta$ is omitted in the notation:
$$0\to\DD^\bullet \to \JJ^\bullet \to (\JJ/\DD)^\bullet\to 0.$$

Our next purpose is to describe the cohomology of the quotient complex $(\JJ/\DD)^\bullet$. Let $CQ$ be the set of cycles of $Q$ and let
$$\CC^n=\bigoplus_{\gamma\in CQ}\JJ_\gamma^n.$$
Recall that for $\gamma\in CQ$,
$$\JJ_\gamma^n= \bigoplus_{\tau\in T_n(\gamma)}\JJ_\tau^n=\bigoplus_{\tau\in T_n(\gamma)}\Hom_k(\tau_\Delta, \ \Delta_\gamma)$$
where $\Delta_\gamma = A_{s(\gamma)} = A_{t(\gamma)}$.
The set of  paths of $Q$ is the disjoint union of $DQ$ and $CQ$, hence at each degree there is a vector space decomposition
$$\JJ^n=\DD^n\oplus \CC^n$$
which provides a vector space identification between $\CC^n$ and $(\JJ/\DD)^n$.

\begin{rema}\label{notsub}
\
\begin{itemize}
\item
The vector spaces $\{\CC^n\}_{n\geq 0}$ are not in general a subcomplex of $\JJ^\bullet$. Indeed, let $\tau$ be a trajectory over a cycle $\gamma$, and let $f_\tau\in \JJ_\tau$. By Proposition \ref{preserve},  $df_\tau\in\bigoplus_{\sigma\in\tau^+}\JJ_\sigma^{n+1}$, where $\tau^+=\tau_0^+\cup \tau_1^+ \cup \tau_2^+$. Observe that  a trajectory  $\sigma\in\tau_1^+$ is not anymore over a cycle since it is obtained by adding an arrow either after the end of $\tau$ or before its beginning  - recall that there are no loops in $Q$.
Moreover, in general $(df_\tau)_\sigma$ is not zero:  the image of $f_\tau$ lies in $A_{s(\omega)}=A_{t(\omega)}$, hence the first and the last summands may be not zero.

\item
Notice that if $\sigma\in\tau_2^+$, then  $(df_\tau)_\sigma=0$. Indeed, the proof of the last item of Proposition \ref{subcomplex} is also valid for cycles.
\item
On the other hand, the trajectories in $\tau_0^+$ are over the same cycle $\gamma$, since they are obtained by increasing by one a waiting time of $\tau$.

\end{itemize}
\end{rema}

Let $\tau$ be a trajectory over a path $\omega$, and let $f_\tau\in\Hom_k\left(\tau_\Delta, \ \Delta_\omega\right)$. Recall that
$$df_\tau = \sum_{\sigma\in \tau^+_0} \left(df_\tau\right)_\sigma + \sum_{\sigma\in \tau^+_1} \left(df_\tau\right)_\sigma + \sum_{\sigma\in \tau^+_2} \left(df_\tau\right)_\sigma.$$

Summarizing, we have proven that the last summand is zero for algebras arising from categories with zero compositions, from now on we will omit it.

\begin{defi}
Let $\gamma\in CQ_m$, and let $\tau$ be a trajectory over $\gamma$ of duration $n$ - hence $n\geq m$.
Let $$d':\CC^{n}\to \CC^{n+1}$$ be the map defined by:
$$d'f_\tau= \sum_{\sigma\in\tau_0^+}(df_\tau)_\sigma.$$
\end{defi}
Notice that the image of $d'$ is indeed contained in $\CC^{n+1}$ since the trajectories of $\tau_0^+$ are over the cycle $\gamma$.
\begin{theo}
The complex of cochains $(\JJ/\DD)^\bullet$ with the induced differential $\overline{d}$ is isomorphic to $(\CC^\bullet, d')$.
\end{theo}

\begin{proof}
Remark \ref{notsub} shows that through the mentioned identification between $(\JJ/\DD)^\bullet$ and $\CC^\bullet$, the coboundary $\overline{d}$ becomes $d'$. \qed
\end{proof}

The subcomplex  $\DD^\bullet$ of $\JJ^{\bullet}$ is a direct sum of subcomplexes indexed by non cycles.  For $\delta\in DQ$ we have defined  its cohomology $\HH^\bullet_\delta$.

We observe that a similar situation is in force for the quotient complex of cochains $((\JJ/\DD)^\bullet, \overline{d})=(\CC^\bullet, d')$, namely this quotient is the direct sum of subcomplexes indexed by $CQ$.

Moreover the differential $d'$ has the same description than the one given in Proposition \ref{dftau} for $d$. This enables us to define globally the cohomology along  an arbitrary path as follows.

\begin{defi}\label{omegacoho}
Let $Q$ be a simply laced quiver with a $Q$-set $\Delta=(A,M)$, and let $\omega\in Q_m$. For $n\geq m$, the \emph{cohomology along $\omega$ of $\Delta$} in degree $n$ is denoted $\HH^n_\omega(\Delta)$ and it is the cohomology of the following complex of cochains $(\KK_\omega^\bullet, \dd)$

$$0\to \Hom_k(M_{a_m}\otimes\dots\otimes M_{a_1},\ \Delta_\omega ) \stackrel{\dd}{\to}\cdots$$$$\stackrel{\dd}{\to}\bigoplus_{\tau \in T_n(\omega)} \Hom_k(\tau_\Delta, \Delta_\omega) \stackrel{\dd}{\to}\bigoplus_{\tau \in T_{n+1}(\omega)} \Hom_k(\tau_\Delta, \Delta_\omega)\stackrel{\dd}{\to}\cdots$$

where $$\dd f_\tau=\sum_{\sigma\in \tau^+_0} (df_\tau)_\sigma.$$
\end{defi}

\begin{rema} Let $Q$ be a simply laced quiver with $Q$-set $\Delta=(A,M)$.
\begin{itemize}

\item If $\omega=\delta$ is a non cycle, then  $(\KK^\bullet_\delta, \dd) = (\JJ^\bullet_\delta, d)$ and the previous definition agrees with Definition \ref{deltacoho}.

\item If $\omega=\gamma$ is a cycle,  then we have seen just before that $(\KK_\gamma^\bullet, \dd)$ is a direct summand of the quotient complex $((\JJ/\DD)^\bullet,\overline{d})=(\CC^\bullet, d')$.

\item If $\omega=x$ is a vertex, that is a cycle of length $0$, then $\KK_x^\bullet$ is the usual complex which computes the Hochschild cohomology of $A_x$. Hence
$$\HH_x^n(\Delta)=\HH\HH^n(A_x).$$

\end{itemize}
\end{rema}

The proof of the following result is clear.

\begin{prop}
Let $\Lambda$ be a $k$-algebra arising from a $k$-category with zero compositions, with simply laced quiver $Q$ and $Q$-set $\Delta=(A,M)$. Let $\JJ^\bullet$ be the complex of cochains of Lemma \ref{byseparable}, which computes $\HH\HH^\bullet(\Lambda_\Delta)$. Let $\DD^\bullet$ be the non cycle subcomplex, let $(\JJ/\DD)^\bullet$ be the quotient, and let $\HH^*(\DD^\bullet)$ and $\HH^*((\JJ/\DD)^\bullet)$ be their respective cohomologies.

The following decompositions hold:

$$\HH^n(\DD^\bullet)=\bigoplus_{\delta\in DQ_{\leq n} } \HH^n_\delta(\Delta),$$

$$\HH^n((\JJ/\DD)^\bullet)=\bigoplus_{\gamma\in CQ_{\leq n} } \HH^n_\gamma(\Delta).$$

\end{prop}
\begin{coro}
For a $k$-algebra arising from a $k$-category with zero compositions as before,  $\HH\HH^*(A)$ is a direct summand of $\HH^*((\JJ/\DD)^\bullet)$.
\end{coro}
As particular cases of the previous proposition, we obtain the following  results in low degrees:

\begin{itemize}
\item $\HH^0(\DD^\bullet)=0,$ since there are no non cycles of length less or equal to zero.

\item
$\HH^1(\DD^\bullet)=\bigoplus_{a\in Q_1}\HH_a^1(\Delta)=\bigoplus_{a\in Q_1}\End_{A_{t(a)}-A_{s(a)}} M_a = \End_{A-A}M,$
by Lemma \ref{HHdelta1}.

\item $\HH^0((\JJ/\DD)^\bullet)= \times_{x\in Q_0} \ZZ A_x = \ZZ A,$  that is the center of $A$. The last item of the previous remark provides the equality.

\item $\HH^1((\JJ/\DD)^\bullet)= \bigoplus_{x\in Q_0}\HH\HH^1(A_x) = \HH\HH^1(A),$
since there are no loops in $Q$, that is $CQ_{\leq 1}= Q_0$, and as a consequence of the last item of the previous remark.

\end{itemize}

\normalsize

\begin{theo}\label{les}
Let $\Lambda_\Delta$ be a $k$-algebra arising from a $k$-category with zero compositions with simply laced quiver $Q$ and $Q$-set $\Delta=(A,M)$. There is a cohomology long exact sequence as follows:
\small
$$
\arraycolsep=2mm\def\arraystretch{1,4}
\begin{array}{lllllllccc}
0 &\to  &0 &\to &\HH\HH^0(\Lambda_\Delta) &\to &\HH\HH^0(A) &\to\\
&&\End_{A-A}M &\to &\HH\HH^1(\Lambda_\Delta) &\to &\HH\HH^1(A) &\to\\
&&\bigoplus_{\delta\in DQ_{\leq 2}}\HH^2_\delta(\Delta) &\to &\HH\HH^2(\Lambda_\Delta) &\to &\HH\HH^2(A)\oplus\bigoplus_{\gamma\in CQ_{2}}\HH^2_\gamma(\Delta)&\to\\
&&\dots\\
&&\bigoplus_{\delta\in DQ_{\leq n}}\HH^n_\delta(\Delta) &\to &\HH\HH^n(\Lambda_\Delta) &\to &\HH\HH^n(A)\oplus\bigoplus_{\substack{\gamma\in CQ_{\leq n}\\\gamma\notin Q_0}}\HH^n_\gamma(\Delta)&\to\\
&&\dots\\

\end{array}$$

\normalsize
\end{theo}

\begin{coro} \label{les nocycles}
If $Q$ has no cycles - that is if $CQ=Q_0$ - the cohomology long exact sequence is as follows:

\small
$$
\arraycolsep=2mm\def\arraystretch{1,4}
\begin{array}{lllllllccc}
0 &\to  &0 &\to &\HH\HH^0(\Lambda_\Delta) &\to &\HH\HH^0(A) &\to\\
&&\End_{A-A}M &\to &\HH\HH^1(\Lambda_\Delta) &\to &\HH\HH^1(A) &\to\\
&&\bigoplus_{\delta\in DQ_{\leq 2}}\HH^2_\delta(\Delta) &\to &\HH\HH^2(\Lambda_\Delta) &\to &\HH\HH^2(A)&\to\\
&&\dots\\
&&\bigoplus_{\delta\in DQ_{\leq n}}\HH^n_\delta(\Delta) &\to &\HH\HH^n(\Lambda_\Delta) &\to &\HH\HH^n(A)&\to\\
&&\dots\\

\end{array}$$

\end{coro}
\normalsize

In Lemma \ref{HHdelta1} we have computed the cohomology along $\delta\in DQ_m$ in its lowest degree:
$$\HH_\delta^m=\Hom_{A_{t(\delta)}-A_{s(\delta)}}(M_\delta, \Delta_\delta).$$
In particular for an arrow $a$,
$$\HH^1_a(\Delta)=\End_{A_{t(a)}-A_{s(a)}} M_a.$$

 In the next section we will show that
 $$\HH^{r+1}_a(\Delta)=\Ext^{r}_{A_{t(a)}-A_{s(a)}} \left( M_a, M_a\right).$$
 For longer paths, we will compute the cohomology under some additional hypotheses, regardless if the path is a non cycle or a cycle.

\section{\sf Cohomology along paths}\label{HHalongpaths}

We recall our setting:  $Q$ is a simply laced quiver,  $\Delta$  a $Q$-set consisting in a set of  algebras $\{A_x\}_{x\in Q_0}$ attached to the vertices of $Q$ and a bimodule $M_a$ for each arrow $a\in Q_1$, where $M_a$ is an $A_{t(a)}-A_{s(a)}$-bimodule.

Our aim is to compute the cohomology $\HH_\omega^*(\Delta)$ along a path $\omega\in Q_m$, see Definition \ref{omegacoho}. In case $m\geq 2$, we will need extra assumptions in order to perform the computation. If $\omega$ is an arrow, no additional assumption is required.

Let $a\in Q_1$ and $n\geq 1$. A trajectory of duration $n$ over $a$ is $\tau_{q,p} = t(a)^{q}, a, s(a)^{p}$, where $p\geq 0$ and $q\geq 0$ verify $q+1+p=n$.

For simplicity, we set  $B=A_{t(a)}$, $A=A_{s(a)}$ and $M=M_a$. Moreover, we omit tensor product symbols over $k$ between vector spaces, and we replace tensors between elements by commas.
Recall that
$$(\tau_{q,p})_\Delta= B^{q} M  A^{p}.$$

Observe that $\left(\tau_{q,p}\right)_0^+$ has  two $(n+1)$-trajectories:
$$\tau_{q+1,p}=t(a)^{q+1}, a, s(a)^{p}\ \ \mbox{ and }\ \ \tau_{q,p+1}=\ t(a)^{q}, a, s(a)^{p+1}.$$

Note that $T_0(a)$ is empty, $T_1(a)$ has one trajectory $t(a)^0, a, s(a)^0$, while $T_2(a)$ has two trajectories, $t(a)^1, a, s(a)^0$ and $t(a)^0, a, s(a)^1$.

These observations lead to the following explicit description of the complex which  provides the cohomology along an arrow:

\begin{lemm}
Given an arrow $a$, the complex of cochains $(\KK^\bullet_a,\dd)$ which computes $\HH_a^*(\Delta)$ is:

$$\begin{array}{lll}
0\longrightarrow&\Hom_k(M,M)\stackrel{\dd_1}{\longrightarrow}\Hom_k(BM,M)\oplus \Hom_k(MA,M)\stackrel{\dd_2}{\longrightarrow}\\
&\Hom_k(B^2M,M)\oplus \Hom_k(BMA,M)\oplus \Hom_k(MA^2,M)\stackrel{\dd_3}{\longrightarrow}\cdots\stackrel{\dd_{n-1}}{\longrightarrow}\\ &\oplus_{p+q+1=n}\Hom_k(B^qMA^p,M)\stackrel{\dd_n}{\longrightarrow}\oplus_{p+q+1=n+1}\Hom_k(B^qMA^p,M)
\stackrel{\dd_{n+1} }{\longrightarrow}\\&\cdots
\end{array}
$$
where for $q+1+p=n$ and $f\in\Hom_k(B^qMA^p,M)$
 $$ \begin{array}{lll}
 (\dd_nf)_{\tau_{q+1,p}}(b_1,\dots, b_{q+1},m,a_1,\dots, a_p)=\\
  &\hskip-3cmb_1f(b_2,\dots,b_{q+1},m,a_1,\dots, a_p)+\\
  &\hskip-3cm\sum_{1}^q(-1)^if(b_1,\dots,b_ib_{i+1},\dots,b_{q+1},m,a_1,\dots, a_p)+\\
  &\hskip-3cm(-1)^{q+1}f(b_1,\dots,b_{q+1}m,a_1,\dots,a_p)
  \end{array}$$
and
 $$ \begin{array}{lll}
 (\dd_nf)_{\tau_{q,p+1}}(b_1,\dots, b_{q},m,a_1,\dots, a_{p+1})=\\
&\hskip-3cm (-1)^q[f(b_1,\dots,b_{q},ma_1,\dots, a_{p+1})+\\
&\hskip-3cm\sum_{1}^{p}(-1)^{i}f(b_1,\dots,b_{q},m,a_1,\dots,a_ia_{i+1},\dots, a_{p+1})+\\
&\hskip-3cm(-1)^{p+1}f(b_1,\dots,b_{q},m,a_1,\dots,a_p)a_{p+1}].
  \end{array}$$
In particular for $f\in \Hom_k(M,M)$:
$$(\dd_1f)_{\tau_{1,0}}(b,m)=bf(m)-f(bm)\ \mbox{ and } \ (\dd_1f)_{\tau_{0,1}}(m,a)=f(ma)-f(m)a.$$

\end{lemm}
We recall the bar resolution over an arbitrary algebra $R$ of a  left $R$-module $X$, by free $R$-modules.
$$\cdots\stackrel{\beta}{\to} R^nX \stackrel{\beta}{\to}R^{n-1}X\stackrel{\beta}{\to}\cdots \stackrel{\beta}{\to}R^2X\stackrel{\beta}{\to}RX\stackrel{\beta}{\to}X\to 0$$
where
$$\begin{array}{llll}
\beta(r_1,\dots, r_n,x)=& \sum_1^{n-1}(-1)^{i+1} (r_1,\dots,r_ir_{i+1},\dots,r_n,x) +\\
& (-1)^{n+1}(r_1\dots,r_{n-1}, r_nx).
\end{array}$$
Note that if $X=R$, then the bar resolution is also a resolution of $R$ as $R$-bimodule.

The next two results have been obtained in \cite{CIBILS2000}, see also \cite{CMRS2002}. We recall their proof for further use. The first one provides a canonical resolution of $M$ as $B-A$-bimodule, which is not its bar resolution over $B\otimes A^{\mathsf{op}}$. It is suitable in order to obtain the complex of cochains $\KK^\bullet_a$ described above, which in turn will enable us to obtain the second result.

\begin{lemm}\label{resolution}
Let $A$ and $B$ be algebras, and let $M$ be a $B-A$-bimodule.
The following complex $\CC^\bullet(M)$ is a free-resolution of $M$ as bimodule,

$$\begin{array}{ll}\displaystyle
\dots\to\bigoplus_{\substack{p+q=n+2\\p,q>0}}B^qMA^p\stackrel{\beta_n}{\to}\bigoplus_{\substack{p+q=n+1\\p,q>0}}B^qMA^p \to&\\
&\hskip-3cm\dots\to BMA^2\oplus B^2MA\stackrel{\beta_1}{\to}BMA\stackrel{\beta_0}{\to}M\to 0
\end{array}$$
where $\beta_0(b,m,a)=bma$. For $n>0$, the differential $\beta_n$ is the sum of
$$\beta_n^{q,p} : B^qMA^p\longrightarrow B^{q-1}MA^p\ \oplus \ B^qMA^{p-1}$$
where $p+q=n+2$ and
\begin{itemize}
\item
if $q\geq2$ and $p\geq 2$,  then $\beta_n^{q,p} $ can be written as the sum of two components ${}_{(1,0)}\beta_n^{q,p}$ and ${}_{(0,1)}\beta_n^{q,p}$  given by

$\begin{array}{lll}
_{(1,0)}\beta_n^{q,p}(b_1,\dots,b_q,m,a_1,\dots,a_p)&=\\
&\hskip-2cm\sum_1^{q-1}(-1)^{i+1}(b_1,\dots,b_ib_{i+1},\dots,b_q,m,a_1,\dots,a_p)+\\&\hskip-2cm(-1)^{q+1}(b_1,\dots,b_{q-1},b_qm,a_1,\dots,a_p)
\end{array}
$

and analogously  for ${}_{(0,1)}\beta_n^{q,p}$,

\item

if $p=1$ and $q\geq 2$, then

$\begin{array}{lll}
\beta_n^{q,1}(b_1,\dots,b_q,m,a_1)=&\\
&\sum_1^{q-1}(-1)^{i+1}(b_1,\dots,b_ib_{i+1},\dots,b_q,m,a_1)+\\&(-1)^{q+1}(b_1,\dots,b_{q-1},b_qm,a_1),
\end{array}
$
\item
and if $p\geq 2$ and $q=1$, the definition of $\beta_n^{1,p}$ is analogous to that of $\beta_n^{q,1}$  in the previous item.

\end{itemize}

\end{lemm}

\begin{proof}
Consider the bar resolution of $M$ as a left $B$-module, and the bar resolution of $A$ as $A$-bimodule:
$$\cdots \to\ BBM \to BM \to 0 \ \mbox{ \ and\  }\  \cdots \to\ AAA \to AA \to 0.$$

Their tensor product over $A$ provides the complex described in the statement. In order to use  the K\"{u}nneth formula (see for instance \cite{WEIBEL}), we first observe that the cycles in each degree of the bar resolution of $A$ are projective left $A$-modules, since the resolution splits as a sequence of left $A$-modules. Hence the tensor product of the bar resolutions has zero homology in positive degrees, while in degree zero its homology is $M\otimes_A A = M$. \qed

\end{proof}

\begin{theo}\label{cohoalongarrow}
Let $Q$ be a simply laced quiver with a $Q$-set $\Delta$, and let $a\in Q_1$. The cohomology of $\Delta$ along $a$ is as follows:

$$\HH_a^{1+r}(\Delta)=\Ext^r_{B-A}(M,M).$$
\end{theo}
\begin{proof}
First we apply the functor $\Hom_{B-A}(-,M)$ to the previous resolution. Let $Y$ and $X$ be $B-A$-bimodules. The canonical isomorphism $$\Hom_{B-A}(BY\!A,X)=\Hom_k(Y,X)$$ provides the complex of cochains $\KK^\bullet_a$.
\qed
\end{proof}

As a consequence, in case there are no paths of length greater than or equal to $2$, the long exact sequence of Corollary \ref{les nocycles} is simpler, as we prove in Corollary \ref{Q2empty}. Note that for this sort of quivers a $k$-algebra arising from a $k$-category is automatically with zero compositions.

\begin{coro}\label{Q2empty}
Let $Q$ be a simply laced quiver with $Q_2$ empty. Let $$\Delta = (A=\times_{x\in Q_0}A_x,\ M=\oplus_{a\in Q_1} M_a)$$ be a $Q$-set and let $\Lambda_\Delta$ be the corresponding algebra arising from a $k$-category. The cohomology long exact sequence is:
\small
$$
\arraycolsep=2mm\def\arraystretch{2}
\begin{array}{lllllllccc}
0 &\longrightarrow  &0 &\longrightarrow &\HH\HH^0(\Lambda_\Delta) &\longrightarrow &\HH\HH^0(A) &\longrightarrow\\
&&\End_{A-A}M &\longrightarrow &\HH\HH^1(\Lambda_\Delta) &\longrightarrow &\HH\HH^1(A) &\longrightarrow\\
&&\Ext^1_{A-A}(M,M)&\longrightarrow &\HH\HH^2(\Lambda_\Delta) &\longrightarrow &\HH\HH^2(A)&\longrightarrow\\
&&\dots\\
&&\Ext^{n-1}_{A-A}(M,M)&\longrightarrow &\HH\HH^n(\Lambda_\Delta) &\longrightarrow &\HH\HH^n(A)&\longrightarrow\\
&&\dots\\

\end{array}$$

\end{coro}
\begin{proof}
There are no cycles of positive length in $Q$, so we consider the long exact sequence of Corollary \ref{les nocycles}. For $n\geq 1$, we have $DQ_{\leq n}= Q_1$. Moreover we have just proven that the cohomology along an arrow $a$ is $\HH_a^{1+r}(\Delta)=\Ext^r_{A_{t(a)}-A_{s(a)}}(M_a, M_a)$. Finally notice that $\bigoplus_{a\in Q_1}\Ext^r_{A_{t(a)}-A_{s(a)}}(M_a, M_a) = \Ext^r_{A-A}(M,M).$
\qed
\end{proof}
\begin{exam}\label{onepoint}
Let $A_y[M]$ be a one point extension, where $A_y$ is a $k$-algebra and $M$ is a left $A_y$-module. By definition $A_y[M]=
\left(
                                                                                                       \begin{array}{cc}
                                                                                                         k & 0 \\
                                                                                                         M & A_y \\
                                                                                                       \end{array}
                                                                                              \right)=\left(
                                                                                                       \begin{array}{cc}
                                                                                                         A_y & M \\
                                                                                                         0 & k \\
                                                                                                       \end{array}
                                                                                                     \right)$.
A one point extension is a $k$-algebra arising from a $k$-category for the quiver $x\cdot \stackrel{a}{\longrightarrow} \cdot y $ and the $Q$-set $(A,M)$ where $A_x=k$ and $M_a=M$. Since there are no paths of length two, the long exact sequence of Corollary \ref{Q2empty} considered for $A_y[M]$ is in force. It coincides with the long exact sequence obtained by D. Happel in \cite{HAPPEL}. Indeed, for a field $k$ we have $\HH\HH^0(k)=k$, while $\HH\HH^n(k)=0$ if $n>0$. The long exact sequence is then as follows:
\small
$$
\arraycolsep=2mm\def\arraystretch{2}
\begin{array}{llclclcllccc}
0 &\longrightarrow  &0 &\longrightarrow &\HH\HH^0(A_y[M]) &\longrightarrow &\HH\HH^0(A_y)\oplus k &\longrightarrow\\
&&\End_{A_y}M &\longrightarrow &\HH\HH^1(A_y[M]) &\longrightarrow &\HH\HH^1(A_y) &\longrightarrow\\
&&\Ext^1_{A_y}(M,M)&\longrightarrow &\HH\HH^2(A_y[M]) &\longrightarrow &\HH\HH^2(A_y)&\longrightarrow\\
&&\dots\\
&&\Ext^{n-1}_{A_y}(M,M)&\longrightarrow &\HH\HH^n(A_y[M]) &\longrightarrow &\HH\HH^n(A_y)&\longrightarrow\\
&&\dots\\

\end{array}$$

\end{exam}
\begin{exam}
Let again $Q$ be the quiver $x\cdot \stackrel{a}{\longrightarrow} \cdot y $. Let $\Delta$ be a $Q$-set. The corresponding algebra arising from a $k$-category is the corner algebra $$\left( \begin{array}{cc}                      A_x & 0 \\
                                                                                                            M_a & A_y \\
                                                                                                          \end{array}
                                                                                                        \right) = \left( \begin{array}{cc}                      A_y & M_a \\
                                                                                                            0 & A_x \\
                                                                                                          \end{array}
                                                                                                        \right).$$ The  long exact sequence of cohomology has been obtained in this case independently in \cite{CIBILS2000}, \cite{MICHELENAPLATZEC} and \cite{GREENSOLBERG}.

\end{exam}

Next we consider a simply laced quiver $Q$ and a path $a_2a_1\in Q_2$. We set $C=A_{t(a_2)}$, $B=A_{s(a_2)}=A_{t(a_1)}$ and $A=A_{s(a_1)}$.

\normalsize

\begin{lemm}
The complex of cochains $(\KK^\bullet_{a_2a_1},\dd)$ which computes $\HH_{a_2a_1}^*(\Delta)$ is as follows:
$$\begin{array}{lll}
0\longrightarrow\Hom_k(M_{a_2}M_{a_1},X)\stackrel{\dd_1}{\longrightarrow}\\\Hom_k(CM_{a_2}M_{a_1},X)\oplus \Hom_k(M_{a_2}BM_{a_1},X)\oplus \Hom_k(M_{a_2}M_{a_1}A,X)\stackrel{\dd_2}{\longrightarrow}\\
\cdots\\
\stackrel{\dd_{n-2}} {\longrightarrow}\bigoplus_{p+q+r+2=n}\Hom_k(C^rM_{a_2}B^qM_{a_1}A^p,X)\stackrel{\dd_{n-1}}{\longrightarrow}\\
\hskip4cm\bigoplus_{p+q+r+2=n+1}
\Hom_k(C^rM_{a_2}B^qM_{a_1}A^p,X)
\stackrel{\dd_{n} }{\longrightarrow}\\\cdots
\end{array}
$$
\end{lemm}
\begin{proof}
 Note that a trajectory $\tau_{r,q,p}\in T_n(a_2a_1)$ is
$$\tau_{r,q,p}= t(a_2)^r,a_2,s(a_2)^q,a_1,s(a_1)^p$$
such that $r+q+p+2=n$. Moreover
$$(\tau_{r,q,p})_\Delta=C^rM_{a_2}B^qM_{a_1}A^p.$$
The stated complex equals the one of Definition \ref{omegacoho} with $\omega=a_2a_1$.
\qed
\end{proof}

\begin{theo} Let $M_{a_1}$ be a $B-A$-bimodule and let $M_{a_2}$ be a $C-B$-bimodule, corresponding to $\omega=a_2a_1\in Q_2$ as previously. If $\Tor^n_B(M_{a_2},M_{a_1})=0$ for $n>0$, then
$$\HH_{a_2a_1}^{2+r}(\Delta)=\Ext^r_{C-A}(M_{\omega}, \Delta_{\omega})$$
where $M_\omega$ is the bimodule defined in \ref{Momega}.
\end{theo}

\begin{proof}
Let $\CC^\bullet(M_{a_2})$ and $\CC^\bullet(M_{a_1})$ be the free resolutions of Lemma \ref{resolution}. In particular, they are right and left $B$-projective resolutions of $M_{a_2}$ and $M_{a_1}$ respectively. Consequently, the homology of  the complex $\CC^\bullet(M_{a_2})\otimes_B\CC^\bullet(M_{a_1})$ is $\Tor^\bullet_B(M_{a_2},M_{a_1})$, see for instance \cite[Theorem 9.3]{MACLANE}. Our assumption insures that this complex is a free $C-A$  resolution of $M_{a_2a_1}$. Applying the functor $\Hom_{C-A}(-,\Delta_{a_2a_1})$ to it, and using the canonical isomorphism $$\Hom_{C-A}(CYA,X)=\Hom_k(Y,X),$$ yields the complex of cochains $\KK_{a_2a_1}^\bullet$.\qed
\end{proof}

\begin{theo}\label{tres} Let $\omega=a_3a_2a_1\in Q_3$. Let $A=A_{s(a_1)}$, $B=A_{s(a_2)}$, $C=A_{s(a_3)}$ and $D=A_{t(a_3)}$.
If $\Tor_n^B(M_{a_2}, M_{a_1})=\Tor_n^C(M_{a_3},M_{a_2a_1})=0$ for $n>0$, then

$$\HH_{\omega}^{3+r}(\Delta)=\Ext^r_{D-A}(M_{\omega}, \Delta_{\omega}).$$

\end{theo}

\begin{proof}
The tensor product of the resolutions of $M_{a_2}$ and $M_{a_1}$ provides as before a free $C-A$  resolution of $M_{a_2a_1}$. In turn, we tensorize it with the resolution of $M_{a_3}$, obtaining a free $D-A$ resolution of $M_{a_3a_2a_1}$. Applying the appropriate functor yields $\KK_{a_3a_2a_1}^\bullet$. \qed
\end{proof}

Now, we consider the general situation.

\begin{defi}
Let $Q$ be a simply laced quiver and let $\Delta$ be a $Q$-set. A path $\omega= a_m\dots a_1$ of length $m\geq 2$ is \emph{Tor-vanishing} if
$$ \Tor^{A_{s(a_i)}}_n(M_{a_{i}}, M_{a_{i-1}\dots a_1})=0$$
for $i=2,\dots, m$ and for all $n>0$.
\end{defi}

The following result is a straightforward generalization of the previous theorem.
\begin{theo}\label{HHwithtorvanishing}
Let $Q$ be a simply laced quiver, let $\Delta$ be a $Q$-set and let $\omega\in Q_m$  with $m\geq 2$, be a Tor vanishing path.
For $r\geq 0$, the following holds
$$\HH_\omega^{m+r}(\Delta)=\Ext^r_{A_{t(w)}-A_{s(w)}}(M_\omega, \Delta_\omega).$$

\end{theo}

\begin{defi}
A $Q$-set is \emph{Tor-vanishing} if all the paths of length $m\geq 2$ are Tor vanishing.
\end{defi}
Observe that if $Q$ is a quiver without cycles and if its maximal paths of length $m\geq 2$ are Tor vanishing, then the $Q$-set is Tor vanishing.

\begin{coro}

Let $Q$ be a simply laced quiver without cycles. Let $\Delta=(A,M)$ be a $Q$-set, and suppose that the maximal paths - excepting arrows - are Tor vanishing. Let $\Lambda_\Delta$ be the corresponding algebra arising from a $k$-category with zero compositions. There is a cohomology  long exact sequence as follows

\small
$$
\arraycolsep=1mm\def\arraystretch{1,2}
\begin{array}{rccclllll}
0 &\to &\HH\HH^0(\Lambda_\Delta) &\to &\HH\HH^0(A) &\to\\
\End_AM &\to &\HH\HH^1(\Lambda_\Delta) &\to &\HH\HH^1(A) &\to\\
\Hom_{A-A}(M^{\otimes_A 2},M)\oplus \Ext_{A-A}^1(M,M) &\to &\HH\HH^2(\Lambda_\Delta) &\to &\HH\HH^2(A)&\to\\
\dots\\
\bigoplus_{r+s=n}\Ext^{r}_{A-A}(M^{\otimes_A s}, M) &\to &\HH\HH^n(\Lambda_\Delta) &\to &\HH\HH^n(A)&\to\\
\dots\\

\end{array}$$

\end{coro}
\begin{proof}
Since $CQ_m=\emptyset$ for $m\neq 0$, we consider the long exact sequence of Corollary \ref{les nocycles}.  For a path $\omega$ of length $m\geq 2$,  Theorem \ref{HHwithtorvanishing} provides
$$\HH_\omega^{m+r}(\Delta)=\Ext^r_{A_{t(w)}-A_{s(w)}}(M_\omega, \Delta_\omega).$$
Observe that $\bigoplus_{\omega\in Q_m}M_\omega = M^{\otimes_A m}$. Moreover if $a$ is an arrow such that $M_a\neq \Delta_\omega$, then $\Ext^r_{A_{t(w)}-A_{s(w)}}(M_\omega, M_a)=0$. We infer
\begin{equation}
\bigoplus_{\omega\in Q_m} \Ext^*_{A_{t(\omega)}-A_{s(\omega)}}(M_\omega, \Delta_\omega)= \Ext^*_{A-A}(M^{\otimes_A m},M).
\end{equation}
\qed
\end{proof}

In case $Q$ has oriented cycles, we denote
$$(M^{\otimes_A m})_D =  \bigoplus_{\delta\in DQ_m}M_\delta \mbox{\  \ and \ } (M^{\otimes_A m})_C = \bigoplus_{\gamma\in CQ_m}M_\gamma.$$
There is a decomposition
\begin{equation}\label{Mm split}
M^{\otimes_A m} = (M^{\otimes_A m})_D \oplus (M^{\otimes_A m})_C.
\end{equation}

\begin{coro}
Let $Q$ be a simply laced quiver and $\Delta=(A,M)$  a Tor vanishing $Q$-set, and let  $\Lambda_\Delta$ be the corresponding algebra arising from a $k$-category with zero compositions. There is a cohomology long exact sequence as follows

\footnotesize
$$
\hskip-0.5cm
\arraycolsep=0.1mm\def\arraystretch{1,4}
\begin{array}{rccclllll}
0 &\to &\HH\HH^0(\Lambda_\Delta) &\to &\HH\HH^0(A) &\to\\
\End_{A-A}M &\to &\HH\HH^1(\Lambda_\Delta) &\to &\HH\HH^1(A) &\to\\
\Hom_{A-A}((M^{\otimes_A 2})_D,M)\oplus \Ext_{A-A}^1(M,M) &\to &\HH\HH^2(\Lambda_\Delta) &\to &\HH\HH^2(A)\oplus \Hom_{A-A}((M^{\otimes_A 2})_C,A)&\to\\
\dots\\
\bigoplus_{r+s=n}^n\Ext^{r}_{A-A}((M^{\otimes_A s})_D, M) &\to &\HH\HH^n(\Lambda_\Delta) &\to &\HH\HH^n(A)\oplus\bigoplus_{r+s=n}\Ext^{r}_{A-A}((M^{\otimes_A s})_C, A) &\to\\
\dots\\

\end{array}$$

\end{coro}

\begin{coro}\label{pro}
Let $Q$ be a simply laced quiver and $\Delta=(A,M)$  a $Q$-set where $M$ is projective as $A$-bimodule. Let $\Lambda_\Delta$ be the corresponding algebra arising from a $k$-category with zero compositions. The cohomology long exact sequence is as follows
$$
\small
\hskip-0.3cm
\def\arraystretch{1,2}
\begin{array}{rccclllll}
0 &\to &\HH\HH^0(\Lambda_\Delta) &\to &\HH\HH^0(A) &\to\\
\End_{A-A}M &\to &\HH\HH^1(\Lambda_\Delta) &\to &\HH\HH^1(A) &\to\\
\Hom_{A-A}((M^{\otimes_A 2})_D,M)&\to &\HH\HH^2(\Lambda_\Delta) &\to &\HH\HH^2(A)\oplus \Hom_{A-A}((M^{\otimes_A 2})_C,A)&\to\\
\dots\\
\Hom_{A-A}((M^{\otimes_A n})_D, M) &\to &\HH\HH^n(\Lambda_\Delta) &\to &\HH\HH^n(A)\oplus\Hom_{A-A}((M^{\otimes_A n})_C, A) &\to\\
\dots\\
\end{array}$$
\end{coro}
\begin{proof}
Recall that $A=\times_{x\in Q_0}A_x$ and $M=\oplus_{a\in Q_1} M_a$ where $M_a$ is an $A_{t(a)}-A_{s(a)}$-bimodule. Hence $M$ is projective as $A$-bimodule if and only if  for every $a\in Q_1$ the bimodule $M_a$ is projective. In turn this implies that the $Q$-set is Tor vanishing. The previous corollary therefore gives the long exact sequence.\qed
\end{proof}
\section{\sf Multiplicative structures}\label{multiplicative}

The Hochschild cohomology of a $k$-algebra $\Lambda$ is an associative algebra with the cup product.  Gerstenhaber in \cite{GERSTENHABER} proved that it is graded  commutative. We recall that if $f'$ and $f$ are cochains of degrees $n'$ and $n$ of the complex of Lemma \ref{byseparable} - for $Z=\Lambda$, their \emph{cup product} $f'\smile f$ is the composition
$$\Lambda^{\otimes_D{(n'+n)}}\cong\Lambda^{\otimes_Dn'}\otimes_D\Lambda^{\otimes_Dn}
\stackrel{f'\otimes_D f}{\xrightarrow{\hspace*{15mm}}}\Lambda\otimes_D\Lambda\longrightarrow\Lambda$$
where the last map is the product of $\Lambda$. The graded Leibniz rule
$$d(f'\smile f) = d(f')\smile f+(-1)^{n'}f'\smile d(f)$$ holds, so the cup product is well defined in cohomology.

Our next purpose is to consider a cup product on the cohomology along paths of a simply laced quiver $Q$ provided with a set $\Delta$.

\begin{defi}\label{cupalongpathsdefi}
 Let $\omega'$ and $\omega$ be paths of $Q$ and let $\tau'$ and $\tau$ be trajectories over them, of durations $n'$ and $n$ respectively. Let $f'_{\tau'}:\tau'_\Delta\to\Delta_{\omega'}$ and $f_\tau:\tau_\Delta\to\Delta_{\omega}$ be cochains of $(\KK_{\omega'}^\bullet, \dd)$ and $(\KK_\omega^\bullet, \dd)$ respectively - see Definition \ref{omegacoho}. Their \emph{cup product} $f'_{\tau'}\smile f_\tau$ is as follows:
\begin{itemize}
\item If $\omega'$ and $\omega$ are not concatenable - that is if $s(\omega')\neq t(\omega)$, then  $f'_{\tau'}\smile f_\tau=0.$
\item If $\omega'$ and $\omega$ are concatenable, let $\tau'\tau$ be the obvious trajectory of duration $n'+n$ over the concatenated path $\omega'\omega$. The \emph{cup product} $f'_{\tau'}\smile f_\tau$ is the composition
$$({\tau'\tau})_\Delta = \tau'_\Delta\otimes \tau_\Delta \stackrel{f'_{\tau'}\otimes f_\tau}{\xrightarrow{\hspace*{15mm}}}\Delta_{\omega'}\otimes\Delta_\omega\longrightarrow\Delta_{\omega'\omega}$$
where the last map is the product in the $k$-algebra  $\Lambda_\Delta$.
\end{itemize}
\end{defi}

\begin{rema}
If the $k$-algebra arises from a $k$-category with zero compositions, and if both $\omega'$ and $\omega$ are not cycles, then the above product map is zero.
\end{rema}

\begin{prop}\label{cup}
The graded Leibniz rule $$\dd(f'\smile f)=\dd (f')\smile f +(-1)^{n'}f'\smile\dd (f)$$ holds, and there is a well defined cup product
$$\HH^{n'}_{\omega'} \otimes \HH^{n}_{\omega}\longrightarrow \HH^{n'+n}_{\omega'\omega}.$$
\end{prop}

The proof of the proposition is the usual one, taking into account the different cases which occur in this setting.

\begin{rema}
 Let $xQx$ be the set of cycles of $Q$ at the vertex $x$. The cup product  defined above provides a $k$-algebra structure on
 $$\bigoplus_{\omega\in xQx}\HH^{\bullet}_{\omega}.$$
 Moreover $\HH\HH^\bullet(A_x)$ is a subalgebra and its complement in the direct sum is a two-sided ideal.
\end{rema}
The proof of the following result is straightforward.
\begin{theo}\label{imagem}
  Let $Q$ be a simply laced quiver, let $\Delta$ be a $Q$-set and let $\Lambda_\Delta$ be the corresponding algebra arising from a $k$-category with zero compositions.
\begin{itemize}
  \item The cup product in  $\HH\HH^n(\Lambda_\Delta)$, restricted to the images of the maps

  % \small
$$
\bigoplus_{\delta\in DQ_{\leq n}}\HH^n_\delta(\Delta) \longrightarrow\HH\HH^n(\Lambda_\Delta)$$

of the cohomology long exact sequence of Theorem \ref{les} is zero. In other words, the cup product annihilates on cocycles which belong to $\DD^\bullet$.
  \item
  The family of maps

$$\HH\HH^n(\Lambda_\Delta) \longrightarrow \HH\HH^n(A)\oplus\bigoplus_{\substack{\gamma\in CQ_{\leq n}\\\gamma\notin Q_0}}\HH^n_\gamma(\Delta)$$

of the cohomology long exact sequence of Theorem \ref{les} provide a graded algebra map.
\end{itemize}

\end{theo}
Our next purpose is to give a formula for the connecting homomorphism
$$ \HH\HH^n(A)\oplus\bigoplus_{\substack{\gamma\in CQ_{\leq n}\\\gamma\notin Q_0}}\HH^n_\gamma(\Delta)\stackrel{\nabla_n}{\xrightarrow{\hspace*{10mm}}} \bigoplus_{\delta\in DQ_{\leq n+1}}\HH^{n+1}_\delta(\Delta)$$
 of the cohomology long exact sequence of Theorem \ref{les}.

 Observe that for each arrow $a$ the identity map $1_{M_a}$ is a canonical element of $\HH^1_a(\Delta)= \End_{A_{t(a)}-A_{s(a)}} M_a$. Moreover, $1_M=\sum_{a\in Q_1}1_{M_a} \in \End_{A-A}M$.

\begin{theo}\label{nabla}
For $m\leq n$, let $\gamma\in CQ_m$ and let $f\in \HH^n_\gamma(\Delta)$. The following holds:
$$\nabla_n (f) = 1_M\smile f + (-1)^{n+1} f\smile 1_M.$$
Moreover, $\nabla_n$ is of degree $1$ with respect to the length of the paths, that is
$$\nabla_n (f) \in \bigoplus_{\delta\in DQ_{m+1}}\HH^{n+1}_\delta(\Delta).$$
\end{theo}
\begin{proof} Let $\tau$ be a trajectory over $\gamma$ of duration $n$. We consider  a cocycle $f_\tau$ in $\CC^n= (\JJ/\DD)^n$. Note that $f_\tau$ has its image contained in $A_{x}$, where $x=s(\gamma)=t(\gamma)$. In order to compute $\nabla_n(f_\tau)$, we view $f_\tau$ as a cochain in $\JJ^n$, thus the Formula (\ref{d}) provides the equalities

$$df_\tau = \sum_{\sigma\in \tau^+} (df_\tau)_\sigma=  \sum_{\sigma\in \tau_1^+} (df_\tau)_\sigma .
$$
Recall that  $\tau_1^+$ is the set of the $n+1$-trajectories
$$t(c)^0, c,s(c)^{p_{m+1}},a_m,s(a_m)^{p_m},\cdots,s(a_2)^{p_2},a_1,s(a_1)^{p_1}$$
where $c\in Q_1$ is any arrow after $\gamma$, joint with the analogous set of trajectories obtained for any arrow $c$ before $\gamma$. Consequently
$$df_\tau= \sum_{c\in xQ_1}(1_{M_c}\smile f_\tau) + (-1)^{n+1} \sum_{c\in Q_1x}(f_\tau\smile 1_{M_c}).$$

We know that $1_M=\sum_{a\in Q_1}1_{M_a}$. Moreover, if $s(\gamma)\neq t(a)$, the map $1_{M_a}\smile f_\tau$ is zero already at the cochain level. These observations lead to the formula.\qed
\end{proof}

\begin{rema}
Let $A_y[M]$ be a one point extension corresponding to the quiver $x\cdot\longrightarrow\cdot y$, see Example \ref{onepoint}. For $n>0$, the connecting homomorphism $$\nabla_n: \HH\HH^n(A_y)\longrightarrow \Ext^{n}_{A_y}(M,M)$$  is as follows. If $f\in \HH\HH^n(A_y)$, then $$\nabla_n(f)= (-1)^{n+1}(f\smile 1_M).$$ Indeed, $1_M\smile f=0$ for no-concatenation reasons. Moreover, since $A_x=k$, the cohomology  $\HH\HH^\bullet (A_x)$ is concentrated in degree $0$, where its value is $k$.
\end{rema}

Through the previous remark, we end this section by linking our work to some of the results of \cite{GREENMARCOSSNASHALL}. In that  paper,  a $k$-algebra map is constructed and it is shown that it coincides with the connecting homomorphism of the long exact sequence. Consequently this map is the above $\nabla$.

It is straightforward to check that the family of maps $\nabla_n$ for one point extensions given by $\nabla_n(f)= (-1)^{n+1}(f\smile 1_M)$ is indeed a graded algebra map.

Another way of considering the same fact is as follows. Let $A$ be a $k$-algebra and let $M$ be a left $A$-module. It is well known that the Hochschild cohomology of the $A$-bimodule $\End_kM$ verifies  $$\HH^n(A, \End_kM)= \Ext^n_A(M,M).$$

Moreover, $\End_kM$ being an $A$-bimodule,  $\HH^\bullet(A, \End_kM)$ is a $\HH\HH^\bullet (A)$-graded bimodule. The point is that for this specific bimodule $\End_kM$  there is a canonical element
$$1_M\in\HH^0(A, \End_kM)= \End_AM.$$
In fact $\End_kM$ is more than an $A$-bimodule, it is also an $A$-module algebra for the composition, that is the structural product map is $A$-balanced and is an $A$-bimodule map. An immediate consequence of this is that  $\HH^n(A, \End_kM)$ is a $k$-algebra, which is canonically isomorphic to the algebra $\Ext^n_A(M,M)$ with the Yoneda product.

\begin{prop}
The morphism
$$\HH\HH^n(A)\longrightarrow \HH^n(A, \End_kM)= \Ext^n_A(M,M)$$
given by the (left or right) action  on the canonical element $1_M$ is a morphism of algebras.
\end{prop}

\section{\sf Square projective algebras}\label{square}

In this section we focus on null-square projective algebras, see for instance \cite{CIBILSREDONDOSOLOTAR2017}.

A \emph{square algebra} is a $k$-algebra arising from a $k$-category built on the following:
\begin{itemize}
\item the round trip quiver $$Q=\ {}_x\cdot \rightleftarrows \cdot_{y}, $$ where the arrow from $x$ to $y$ is denoted $a$ and the reverse one is denoted $b$,
\item
a $Q$-set $(\AA,\MM)$  given by
\begin{itemize}
  \item the algebras $A_x=A$, $A_y=B$ and the \emph{corner bimodules} $M_a=M$ and $M_b=N$, hence $\AA=A\times B$ and $\MM= M\oplus N$,
  \item the bimodule maps
$$\alpha : N\otimes_B M\to A \mbox{ \ and \ } \beta : M\otimes_A N \to B$$
verifying \emph{associativity constraints}, that is the following diagrams commute:
\small
 \begin{equation}\label{assoc}
\xymatrix@!C{
M \otimes_A N \otimes_B M    \ar@{->}[d]_{\beta\otimes 1}      \ar[r]^-{1\otimes \alpha}       &        M\otimes_A A        \ar@{->}[d] \\
{}B\otimes_B M              \ar[r]               &    M}
\xymatrix@!C{
N \otimes_B M \otimes_A N    \ar@{->}[d]_{\alpha\otimes 1}      \ar[r]^-{1\otimes \beta}       &        N\otimes_B B        \ar@{->}[d] \\
{}A\otimes_A N              \ar[r]               &    N}
\end{equation}
\end{itemize}
\end{itemize}
\normalsize
The square algebra is   $\left(
                                                                                                          \begin{array}{cc}
                                                                                                            A & N \\
                                                                                                            M & B \\
                                                                                                          \end{array}
                                                                                                        \right)$
 with matrix multiplication such that if $m\in M$ and $n\in N$, then
$$\left(
 \begin{array}{cc}
 0 & 0 \\
 m& 0 \\
  \end{array}
  \right)
  \left(
 \begin{array}{cc}
 0 & n \\
 0& 0 \\
  \end{array}
  \right)=
  \left(
 \begin{array}{cc}
 0 & 0 \\
 0& \beta(m\otimes n) \\
  \end{array}
  \right)
    \mbox{\  and}$$ $$
  \left(
 \begin{array}{cc}
 0 & n \\
 0& 0 \\
  \end{array}
  \right)
  \left(
 \begin{array}{cc}
 0 & 0 \\
 m& 0 \\
  \end{array}
  \right)=
  \left(
 \begin{array}{cc}
 \alpha(n\otimes m) & 0 \\
 0& 0  \\
  \end{array}
  \right).$$

We note that a $k$-algebra $\Lambda$ with a chosen idempotent $e$ provides a square algebra
$\Lambda = \left(
             \begin{array}{cc}
               e\Lambda e & e\Lambda f \\
               f\Lambda e & f\Lambda f \\
             \end{array}
           \right)$, where $f=1-e$.

\begin{defi}
\
\begin{itemize}
\item A \emph{square projective algebra} is a square algebra such that $\MM$ is a projective $\AA$-bimodule, or, equivalently, the corner bimodules are projective.
\item A \emph{null-square algebra} is a square algebra such that $\alpha=\beta=0$.
  \item A \emph{null-square projective algebra} is a square algebra verifying both previous requirements.
 \end{itemize}
\end{defi}

\begin{prop}
Let $\Lambda$ be a null-square projective algebra. For $m>0$, there is a five-term exact sequence as follows:

$$0\longrightarrow\HH\HH^{2m}(\Lambda)\longrightarrow\HH\HH^{2m}(\AA)\oplus\Hom_{\AA-\AA}\left(\MM^{\otimes_\AA 2m}, \AA\right)\stackrel{\nabla_{2m}}\longrightarrow$$
$$\Hom_{\AA-\AA}\left(\MM^{\otimes_\AA {2m+1}}, \MM\right)\longrightarrow \HH\HH^{2m+1}(\Lambda)\longrightarrow\HH\HH^{2m+1}(\AA)\longrightarrow 0.$$
For $m=0$ the exact sequence:
$$0\longrightarrow \HH\HH^0(\Lambda)\longrightarrow \HH\HH^0(\AA)\stackrel{\nabla_0}{\longrightarrow} \End_{\AA-\AA}\MM \longrightarrow \HH\HH^1(\Lambda) \longrightarrow \HH\HH^1(\AA)\longrightarrow 0.$$
\end{prop}

\begin{proof}
A null-square algebra is built on the round trip quiver $$Q=\cdot \rightleftarrows \cdot$$ where  cycles are precisely the paths of even length.
 If $n$ is odd, then $(\MM^{\otimes_\AA n})_C=0$, and  $(\MM^{\otimes_\AA n})_D= \MM^{\otimes_\AA n}$, while if $n>0$ is even, then $(\MM^{\otimes_\AA n})_D=0$ and $(\MM^{\otimes_\AA n})_C= \MM^{\otimes_\AA n}$. Corollary \ref{pro} applies since $\MM$ is a projective bimodule. The previous observations show that the cohomology long exact sequence of Corollary \ref{pro} splits into five-term exact sequences. \qed
\end{proof}

Let $\nabla'_{2m}: \Hom_{\AA-\AA}\left(\MM^{\otimes_\AA 2m}, \AA\right)\longrightarrow
\Hom_{\AA-\AA}\left(\MM^{\otimes_\AA {2m+1}}, \MM\right)$
 be the restriction of $\nabla_{2m}$ to the second direct summand in the above five-term exact sequence.

\begin{theo}\label{HHnullssquare}
Let $\Lambda$ be a null-square projective algebra as before. For $m>0$, the following holds:
\begin{itemize}
\item
$\HH\HH^{2m}(\Lambda)= \HH^{2m}(\AA)\oplus \Ker \nabla'_{2m}.$
\item
There is a short exact sequence
$$0\to \Coker\nabla'_{2m} \to \HH^{2m+1}(\Lambda)\to \HH^{2m+1}(\AA)\to 0.$$
\end{itemize}
\end{theo}
\begin{proof}
We assert that $\nabla_{2m}$ restricted to $\HH\HH^{2m}(\AA)$ is zero, for $m>0$. Indeed, we know that $\nabla_{2m}$ is of degree $1$ with respect to the length of paths, see Theorem \ref{nabla}. Hence, by Theorem \ref{cohoalongarrow}
$$\nabla_{2m}(\HH\HH^{2m}(\AA))\ \subset \ \HH\HH^{2m+1}_a(\Delta)\oplus \HH\HH^{2m+1}_b(\Delta)=\Ext^{2m}_{\AA-\AA}(\MM,\MM)=0$$
since $\MM$ is projective. Consequently,  $\Ker \nabla_{2m}= \HH^{2m}(\AA)\oplus \Ker \nabla'_{2m}$. This provides the required decomposition of $\HH\HH^{2m}(\Lambda)$ since the first map of the previous five-term exact sequence is injective.

Moreover $\Coker \nabla_{2m} = \Coker \nabla'_{2m}$, and the exact sequence of the statement follows. \qed

\end{proof}

\begin{coro}\label{prodofcoho}
Let $\Lambda=\left(
                                                                                                          \begin{array}{cc}
                                                                                                            A & N \\
                                                                                                            M & B \\
                                                                                                          \end{array}
                                                                                                        \right)$ be a null-square projective algebra.

\begin{itemize}
\item
If $n>0$ and $\Lambda$ is finite dimensional,  then $$\dim\HH\HH^n(\Lambda)\geq \dim\HH\HH^n(A)+\dim\HH\HH^n(B).$$
\item
If there is a positive integer $h$ such that $\MM^{\otimes_\AA h}=0$, then for $n\geq h$ $$\HH\HH^n(\Lambda) = \HH\HH^n(A)\oplus\HH\HH^n(B).$$
\end{itemize}
\end{coro}
\begin{proof}
Since $\AA=A\times B$, we have $\HH\HH^n(\AA)=\HH\HH^n(A)\oplus\HH\HH^n(B).$ The previous theorem provides the results.\qed
\end{proof}

Next we will consider square projective algebras where the corner bimodules are free bimodules of rank one, that is $M=BA$ and $N=AB$ - recall that we drop tensor product symbols over $k$.  We will first prove that the bimodule morphisms $\alpha$ and $\beta$  are necessarily zero, except if $A=B=k$. Next we will show that   $\nabla'_{2m}$ is injective for $m>0$.

\begin{theo}\label{freerankonezero}
Let $\Lambda=\left(
                                                                                                          \begin{array}{cc}
                                                                                                            A &  AB\\
                                                                                                           BA  & B \\
                                                                                                          \end{array}
                                                                                                        \right)$
                                                                                                        be a square algebra where the corner bimodules are free of rank one. The algebra $\Lambda$ is a null-square algebra, except if $A=B=k$.
\end{theo}
\begin{proof}
Let $\alpha : N\otimes_B M\to A$ and  $\beta : M\otimes_A N \to B$ be the bimodule maps of the set verifying the associativity constraints (\ref{assoc}) for $M=BA$ and $N=AB$. We will prove that $\alpha=\beta=0$.
Notice that
\small
$$\Hom_{A-A}(N\otimes_B M, A)=\Hom_{A-A}(AB\otimes_B BA, A)= \Hom_{A-A} (ABA, A)= \Hom_k(B,A).$$
\normalsize
Let $\overline{\alpha}\in \Hom_k(B,A)$ be the linear map corresponding to $\alpha$ through the composition of the previous canonical isomorphisms. Similarly, let $\overline{\beta}\in\Hom_k(A,B)$ be the linear map corresponding to $\beta$. A simple computation shows that the associativity constraints (\ref{assoc}) are equivalent to the following  for every $a\in A$ and $b\in B$:
\begin{eqnarray}
 \label{associativeone} 1\otimes a\overline{\alpha}(b) &=& \overline{\beta}(a)b\otimes 1 \ \in BA \\
 \label{associativetwo}1\otimes b\overline{\beta}(a) &=&  \overline{\alpha}(b)a\otimes 1 \ \in AB.
\end{eqnarray}
Let $A'$ and $B'$ be vector subspaces of $A$ and $B$, such that $A=k\oplus A'$ and $B=k\oplus B'$. For $a\in A$ and $b\in B$, let $a=a_1+a'$ and $b=b_1+b'$ be the corresponding decompositions.

The equality (\ref{associativeone}) for $a=b=1$ gives $1\otimes \overline{\alpha}(1)= \overline{\beta}(1)\otimes 1$, then
$$1\otimes \overline{\alpha}(1)_1 + 1\otimes \overline{\alpha}(1)'= \overline{\beta}(1)_1\otimes 1 + \overline{\beta}(1)'\otimes 1.$$
The tensors $1\otimes \overline{\alpha}(1)_1$ and $\overline{\beta}(1)_1\otimes 1$ are both in $k\otimes_k k$, while $1\otimes \overline{\alpha}(1)'\in k\otimes_k A'$ and $\overline{\beta}(1)'\otimes 1\in B'\otimes_k k$ belong to different direct summands of $BA$, which implies that they are both zero. Moreover $\overline{\alpha}(1)'=0=\overline{\beta}(1)'.$
Consequently there is $\lambda\in k$ such that
$$\overline{\alpha}(1)=\overline{\alpha}(1)_1=\lambda=\overline{\beta}(1)_1= \overline{\beta}(1).$$

For all $a\in A$ and $b=1$, the equality (\ref{associativeone})  gives $1\otimes a\overline{\alpha}(1)=\overline{\beta}(a)\otimes 1$, hence
\begin{equation}\label{adelante}
1\otimes \lambda a = \overline{\beta}(a)\otimes 1
\end{equation}
that is
$$1\otimes \lambda a_1 + 1\otimes \lambda a' = \overline{\beta}(a)_1\otimes 1 + \overline{\beta}(a)'\otimes 1.$$
If $\lambda\neq 0$, then $a'=0$ for every $a\in A$; the same computation for the other associativity constraints provides $b'=0$ for every $b\in B$, that is $A=B=k$.
If $\lambda=0$ we infer from (\ref{adelante}) that $\overline{\beta}=0$; similarly $\overline{\alpha}=0$.\qed

\end{proof}

\begin{prop}\label{injective}
Let $\Lambda$ be as in Theorem \ref{freerankonezero}. The morphism $\nabla'_{2m}$ is injective for $m>0$, except if $A=k$ and $B=k$.
\end{prop}
\begin{proof}
From the previous result, $\Lambda$ is a null-square projective algebra. We first pay attention to the general case where the corner projective bimodules $M$ and $N$ are not necessarily free of rank one. Consider
  $$\nabla'_{2m}: \Hom_{\AA-\AA}\left(\MM^{\otimes_\AA 2m}, \AA\right)\rightarrow
\Hom_{\AA-\AA}\left(\MM^{\otimes_\AA {2m+1}}, \MM\right)$$
and the vector space decomposition
$$\Hom_{A-A}\left((N\otimes_B M)^{\otimes_A m}, A\right)\oplus \Hom_{B-B}\left((M\otimes_A N)^{\otimes_B m},B\right)
\ \stackrel{\nabla'_{2m}}{\xrightarrow{\hspace{15mm}}}$$
$$\Hom_{B-A}\left(M\otimes_A (N\otimes_B M)^{\otimes_A m}, M\right)
\oplus
\Hom_{A-B}\left((N\otimes_B M)^{\otimes_A m}\otimes_A N, N \right).$$

Let $\left[\nabla'_{2m}\right]_M$  and $\left[\nabla'_{2m}\right]_N$ be the  components of $\nabla'_{2m}$ with values in the first and in the second target summands. Hence
$$\Ker \nabla'_{2m} = \Ker \left[\nabla'_{2m}\right]_M \cap \Ker \left[\nabla'_{2m}\right]_N.$$

Moreover, for $(\varphi, \psi)$ in the source vector space, Theorem \ref{nabla} provides
$$\left[\nabla'_{2m}\right]_M(\varphi, \psi)= 1_M\otimes \varphi - \psi\otimes 1_N.$$

Then $$\Ker \left[\nabla'_{2m}\right]_M= \{(\varphi, \psi) \ \mid \ 1_M\otimes \varphi = \psi\otimes 1_N\}.$$

In other words, let
$$\mu: \Hom_{A-A}\left((N\otimes_B M)^{\otimes_A m}, A\right)\to \Hom_{B-A}\left(M\otimes_A (N\otimes_B M)^{\otimes_A m}, M\right)$$
$$\nu : \Hom_{B-B}\left((M\otimes_A N)^{\otimes_B m},B\right)\to \Hom_{B-A}\left(M\otimes_A (N\otimes_B M)^{\otimes_A m}, M\right)$$
be defined by $\mu(\varphi) = 1_M\otimes \varphi $ and $\nu (\psi) = \psi\otimes 1_M$. Notice that $\Ker \left[\nabla'_{2m}\right]_M$ is the pullback of $\mu$ and $\nu$.

For  $M=BA$ and $N=AB$ there are canonical identifications:
\begin{itemize}

 \item $(N\otimes_BM)^{\otimes_A m}=A(BA)^m$

 \item $M\otimes_A (N\otimes_B M)^{\otimes_A m}= (BA)^{m+1}$

 \item $\Hom_{A-A}\left((N\otimes_B M)^{\otimes_A m}, A\right)= \Hom_k ( (BA)^{m-1}B, A).$

 \item $\Hom_{B-A}\left(M\otimes_A (N\otimes_B M)^{\otimes_A m}, M\right)= \Hom_k(A(BA)^{m-1}B, BA).$

\end{itemize}
For simplicity we set $X= (BA)^{m-1}$. Through the previous identifications, $\mu$ corresponds to a linear map:
$$\overline{\mu}: \Hom_k (XB, A)\to \Hom_k(AXB, BA)$$
$$\overline{\mu}(f)(a\otimes x\otimes b)= 1\otimes af(x\otimes b).$$
Similarly we obtain:
$$\overline{\nu} : \Hom_k(AX,B)\to \Hom_k(AXB, BA)$$
$$[\overline{\nu}(g)](a\otimes x\otimes b)= g(a\otimes x)b\otimes 1.$$
We suppose that for all $a\in A$, $x\in X$ and $b\in B$
$$1\otimes af(x\otimes b)= g(a\otimes x)b\otimes 1$$
which is similar to the equality (\ref{associativeone}).  Computing $\Ker \left[\nabla'_{2m}\right]_N$ leads to the analogous result, which is similar to (\ref{associativetwo}). Computations equal to those in the proof of Theorem \ref{freerankonezero} show that $f=g=0$, except when $A=k$ and $B=k$.\qed
\end{proof}

\begin{coro}
Let $A$ and $B$ be finite dimensional  algebras, and let $\Lambda=\left(
                                                                                                          \begin{array}{cc}
                                                                                                            A &  AB\\
                                                                                                           BA  & B \\
                                                                                                          \end{array}
                                                                                                        \right)$
                                                                                                        be the square algebra where the corner bimodules are free of rank one.

Except if $A=B=k$, the following hold
\begin{itemize}

\item
$ \HH\HH^0(\Lambda)= k\times k,$

\item

$ \dim \HH\HH^1(\Lambda) =
\dim \HH\HH^{1}(A) + \dim \HH\HH^{1}(B)   - (\dim \HH\HH^0 A + \dim \HH\HH^0 B) +2(\dim A \dim B +1),$

\item
$ \HH\HH^{2m}(\Lambda)=\HH\HH^{2m}(A)\oplus \HH\HH^{2m}(B)$ for $m>0$,

\item
$ \dim \HH\HH^{2m+1}(\Lambda) =
  \dim \HH\HH^{2m+1}(A) + \dim \HH\HH^{2m+1}(B)+\\2(\dim A \dim B)^m(\dim A \dim B -1)$ for $m>0$.
\end{itemize}
\end{coro}

\begin{proof}
The center of a square algebra  $\Lambda=\left(
                                                                                                          \begin{array}{cc}
                                                                                                            A & N \\
                                                                                                            M & B \\
                                                                                                          \end{array}
                                                                                                        \right)$
 is
$$\{(a,b)\in \HH\HH^0(A)\times  \HH\HH^0(B) \mid \mbox{\ for all \ } m\in M\   n\in N, \ bm=ma \mbox{ \ and\ } an=nb\}.$$

If $M=BA$ and $N=AB$, we infer  $ \HH\HH^0(\Lambda)= k\times k$.

Note that, as vector spaces, $\End_{B-A}BA = BA$ and  $\End_{A-B}AB = AB$. The five-term exact sequence  of Theorem \ref{HHnullssquare} for $m=0$ is
$$0\to k\times k \to \HH\HH^0(A)\oplus \HH\HH^0(B) \stackrel{\nabla_0}{\to} BA\oplus AB \to \HH\HH^1(\Lambda) \to \HH\HH^1(A)\oplus\HH\HH^1(B)\to 0.$$
Hence
\small
$$2-\dim \HH\HH^0 A - \dim \HH\HH^0 B + 2 \dim A\dim B - \dim \HH\HH^1(\Lambda) + \dim \HH\HH^{1}(A) + \dim \HH\HH^{1}(B) =0.$$
\normalsize
The two  equalities for $m>0$ are a consequence of Theorem \ref{HHnullssquare} and of Proposition \ref{injective}. Indeed,
\vskip2mm

$\begin{array}{llll}
\dim  \Hom_{B-A}\left(M\otimes_A (N\otimes_B M)^{\otimes_A m}, M\right)&=&\dim\Hom_k(A(BA)^{m-1}B, BA) =\\
&&(\dim A \dim B)^{m+1}\\
\\
\dim  \Hom_{A-B}\left(N\otimes_B (M\otimes_A N)^{\otimes_B m}, N\right)&=&\dim\Hom_k(B(AB)^{m-1}A, AB)=\\&&(\dim A \dim B)^{m+1}
\end{array}$

and
\vskip3mm
$\begin{array}{llll}
\dim \Hom_{A-A}\left((N\otimes_B M)^{\otimes_A m}, A\right)&=&\dim\Hom_k ( (BA)^{m-1}B, A)=\\
&&(\dim A \dim B)^m\\
\\
\dim\Hom_{B-B}\left((M\otimes_A N)^{\otimes_B m}, B\right)&=&\dim\Hom_k ( (AB)^{m-1}A, B)=\\
&&(\dim A \dim B)^m.
\end{array}$

\noindent Since $\nabla'_{2m}$ is injective, we obtain:
$$\dim \coker\nabla'_{2m}= 2(\dim A \dim B)^{m+1} - 2(\dim A \dim B)^m.$$\qed
\end{proof}

\begin{coro}
Let $\Lambda$ be a $k$-algebra as in the previous result.
\begin{itemize}
  \item $\HH\HH^{2m+1}(\Lambda)\neq 0$ for all $m$.
  \item For $m>0$, $\HH\HH^{2m}(\Lambda)=0$ if and only if $\HH\HH^{2m}(A)=0 = \HH\HH^{2m}(B)$.
\end{itemize}

\end{coro}

\begin{rema}
Let $A=kQ_A/I$ and $B=kQ_B/J$ be finite dimensional algebras where $Q_A$ and $Q_B$ are finite quivers, and $I$ and $J$ are admissible ideals - hence $Q_A$ and $Q_B$ are the Gabriel's quivers of $A$ and $B$.

Let $\Lambda=\left(
                                                                                                          \begin{array}{cc}
                                                                                                            A &  AB\\
                                                                                                           BA  & B \\
                                                                                                          \end{array}
                                                                                                        \right)$
                                                                                                        be a square algebra where the corner bimodules are free of rank one.

The Gabriel quiver $Q_\Lambda$ of $\Lambda$ is the disjoint union of $Q_A$ and  $Q_B$, with in addition new arrows as follows: one arrow from each vertex of $Q_A$ to each vertex of $Q_B$, and conversely. Let $K$ be the admissible ideal of $kQ_\Lambda$  generated by $I$, $J$, and all the paths containing two new arrows. Then $\Lambda = kQ_\Lambda /K$.
\end{rema}

\begin{exam}
Let $A=kQ_A$ and $B=kQ_B$ be finite dimensional path algebras, where $Q_A$ and $Q_B$ are connected quivers without oriented cycles which are not both reduced to one vertex without arrows - that is $A=B=k$ is not considered.

Let $Q$ be the quiver as in the previous remark, that is $Q$ is the disjoint union of $Q_A$ and  $Q_B$, with in addition new arrows joining all the vertices of $Q_A$ to all the vertices of $Q_B$ and \emph{vice versa}. Let $K$ be the two sided ideal of $kQ$ generated by the paths of the form $v\omega u$ where $\omega$ is a path of $Q_A$ or of $Q_B$, and $u$ and $v$ are new arrows. Let $\Lambda=kQ/K$. Then
\begin{itemize}
  \item
$\dim \HH\HH^0(\Lambda)=2,$
\item
$\dim \HH\HH^1(\Lambda)=\dim \HH\HH^1(A)+\dim \HH\HH^1(B)+ 2 \dim A \dim B,$
\item
$\dim \HH\HH^{2m}(\Lambda)=0$, for $m>0$,
\item
$\dim \HH\HH^{2m+1}(\Lambda) =  2(\dim A \dim B)^m(\dim A \dim B -1)$, for $m>0$.
\end{itemize}

Indeed, the Hochschild cohomology of a path algebra vanishes in degrees $2$ and higher.
\end{exam}

\section{\sf Square projective algebras via Peirce quivers}\label{Peirce}

The main purpose of this section is to describe square algebras $\Lambda=\left(
                                                                                                          \begin{array}{cc}
                                                                                                            A & N \\
                                                                                                            M & B \\
                                                                                                          \end{array}
                                                                                                        \right)$
which have the property that  $$\HH\HH^n(\Lambda)=\HH\HH^n(A)\oplus \HH\HH^n(B)$$ for $n$ large enough.

Recall that a system $E$ of a $k$-algebra $\Lambda$ is a finite set of complete orthogonal idempotents - not necessarily primitive.

Let $\Lambda$ be a $k$-algebra arising from a $k$-category built on a simply laced quiver with a $Q$-set $\Delta$. As already mentioned, the Peirce $Q_0$-quiver of $\Lambda_\Delta$ is  $Q$.

In this section we will consider null-square projective algebras   $\Lambda=\left(
                                                                                                          \begin{array}{cc}
                                                                                                            A & N \\
                                                                                                            M & B \\
                                                                                                          \end{array}
                                                                                                        \right)$
where the projective corner bimodules will be given through systems of $A$ and $B$. Let $E$ and $F$ be systems of the algebras $A$ and $B$, respectively. If $e\in E$ and $f\in F$, then $Bf\otimes eA$ is a projective $B-A$-bimodule. Hence for integers ${}_fm_e\geq 0$, the bimodule
$$M=\bigoplus_{e\in E, f\in F} {}_fm_e (Bf\otimes eA) $$
is also projective. Note that $M$ is free of rank one if and only if all the integers ${}_fm_e$ are  $1$.

The proof of the following is straightforward.

\begin{lemm}\label{Peirceofnullsquareproj}
Let $A$ and $B$ be algebras provided with systems $E$ and $F$, and let $Q_E$ and $Q_F$ be their corresponding Peirce quivers.

Let $M=\bigoplus_{e\in E, f\in F} {}_fm_e (Bf\otimes eA) $ and $N=\bigoplus_{e\in E, f\in F} {}_en_f (Ae\otimes fB)$ be projective bimodules and let $\Lambda=\left(
                                                                                                          \begin{array}{cc}
                                                                                                            A & N \\
                                                                                                            M & B \\
                                                                                                          \end{array}
                                                                                                        \right)$
be the corresponding null-square projective algebra, that is the algebra arising from a $k$-category with zero compositions built on the round trip quiver with set the algebras $A$ and $B$ on the vertices and the bimodules $M$ and $N$ on the arrows. We set $\AA=A\times B$ and $\MM=M\oplus N$.

The set $E\cup F$ is a system of $\Lambda$, and the Pierce $E\cup F$-quiver $Q_{E\cup F}$ of $\Lambda$ is the disjoint union of $Q_E$ and $Q_F$ - we view these quivers as located in horizontal up and down plans. In addition there is a \emph{(vertical down)} arrow from $e$ to $f$ if ${}_fm_e\neq 0$ and a \emph{(vertical up)} arrow from $f$ to $e$ if ${}_en_f\neq 0$, see Figure \ref{figure} below.
\end{lemm}

\begin{defi}
In the situation of the previous lemma, an \emph{efficient path} of $Q_{E\cup F}$ is a path of $Q_{E\cup F}$ which does not contain two successive arrows of $Q_E$, nor of $Q_F$.

\end{defi}

\begin{theo}
In the setting of Lemma \ref{Peirceofnullsquareproj}, there exists $h$ such that $\MM^{\otimes_\AA h}=0$ if and only if $Q_{E\cup F}$ has  no efficient cycles.
\end{theo}
\begin{proof} We  first assert that for $h\geq 2$, we have   $\MM^{\otimes_\AA h}\neq 0$ if and only if there exists an efficient path which has $h$ vertical arrows.

The following decomposition holds:
\vskip1,5mm
$\begin{array}{ll}
N\otimes_B M= &\bigoplus_{\substack{e,e'\in E\\ f,f'\in F}} {}_{e'}n_{f'}\ {}_fm_e \left(Ae'\otimes f'B\otimes_BBf\otimes eA\right)=\\
\\
&\bigoplus_{\substack{e,e'\in E\\ f,f'\in F}} {}_{e'}n_{f'}\ {}_fm_e \left(Ae'\otimes f'Bf\otimes eA\right).
\end{array}$
\vskip1,5mm
Note that if the direct summand $\left(Ae'\otimes f'Bf\otimes eA\right)$ is not zero, then there is an efficient path from $e$ to $e'$ which contains two vertical arrows, namely  from $e$ to $f$ and from $f'$ to $e'$. If $f\neq f'$, then there is an arrow in $Q_B$ corresponding to $f'Bf\neq 0$, see Figure \ref{figure}. If $f=f'$, then the vertical arrows are concatenated.

Conversely, if there is an efficient path starting at a vertex $e_-\in E$ and  ending at a vertex $e'_+$ of $E$  which has two vertical arrows, then there is a direct summand of  $N\otimes_B M$ which is non zero.

\renewcommand{\figurename}{Fig.}

\begin{figure}[h]
\begin{center}
\begin{tikzpicture}[
        x={(-0.35cm,-0.35cm)},
        y={(0.7cm,0cm)},
        z={(0cm,0.9cm)},
        font=\tiny,
        ]
  \fill[black!7] (-4,0,0) -- (1,0,0) -- (1,6,0) -- (-4,6,0) -- cycle;
  \fill[black!7] (-4,0,2) -- (1,0,2) -- (1,6,2) -- (-4,6,2) -- cycle;
  \draw[thick] (-4,0,0) -- (1,0,0) -- (1,6,0);
  \draw[thick] (-4,0,2) -- (1,0,2) -- (1,6,2);
  \coordinate (f') at (-1,3,0);
  \coordinate (f) at (-2,4,0);
  \coordinate (e'_+) at (-2,2,2);
  \coordinate (e_-) at (-3,3,2);
  \coordinate (e') at (-1,3,2);
  \coordinate (e) at (-2,4,2);
  \foreach \n in {f', f, e, e'}
    {
    \fill (\n) circle(1.5pt);
    \node[below right] at (\n) {$\n$};
    }
  \foreach \n in {e'_+, e_-}
    {
    \fill (\n) circle(1.5pt);
    \node[above left] at (\n) {$\n$};
    }
  \begin{scope}[
        shorten <=3pt,
        shorten >=3pt,
        thick
        ]
  \foreach \a / \b in {f/f', e'/e'_+, e_-/e}
    \draw[->] (\a) -- (\b);
  \foreach \a / \b in {f'/e'}
    {
    \draw[shorten >=0pt] (\a) -- (\a |- 1,0,2);
    \draw[->, shorten <=0pt, densely dashed] (\a |- 1,0,2) -- (\b);
    }
  \foreach \a / \b in {e/f}
    {
    \draw[shorten >=0pt, densely dashed] (\a) -- (\a |- 1,0,2);
    \draw[->, shorten <=0pt] (\a |- 1,0,2) -- (\b);
    }
  \node at (0.5,0.25,0) {$B$};
  \node at (0.5,0.25,2) {$A$};
  \end{scope}
\end{tikzpicture}
\end{center}
\caption{}
\label{figure}
\end{figure}

Hence, $N\otimes_B M\neq 0$ if and only if there exists an efficient path starting and ending at vertices of $E$, and containing two vertical arrows. The analogous statement holds for $M\otimes_A N.$

Since $\MM\otimes_\AA \MM = (N\otimes_B M) \oplus (M\otimes_A N)$, the assertion is proved for $h=2$. For arbitrary $h$, the proof follows by the same type of considerations.

If there are no efficient cycles, then the length of the efficient paths is bounded, since the number of vertical arrows is finite. Hence there exists $h$ such that $\MM^{\otimes_{\AA}h}=0$. Conversely, if $\MM^{\otimes_{\AA}h}=0$ for some $h$, then $\MM^{\otimes_{\AA}n}=0$ for all $n\geq h$. However, if there is an efficient cycle with $l$ vertical arrows, then there are efficient paths with $rl$ vertical arrows for any positive integer $r$, and $\MM^{\otimes_\AA rl}\neq 0$ for all positive integers $r$. \qed
\end{proof}
\begin{theo}\label{noefficientcycles}
Let $A$ and $B$ be algebras provided with systems $E$ and $F$ respectively. Let $M=\bigoplus_{e\in E, f\in F} {}_fm_e (Bf\otimes eA) $ and $N=\bigoplus_{e\in E, f\in F} {}_en_f (Ae\otimes fB)$ be projective bimodules and let $\Lambda=\left(
                                                                                                          \begin{array}{cc}
                                                                                                            A & N \\
                                                                                                            M & B \\
                                                                                                          \end{array}
                                                                                                        \right)$
be the corresponding null-square projective algebra. We set $\AA=A\times B$ and $\MM=M\oplus N$. Suppose there are no efficient cycles in the $E\cup F$-Peirce quiver of $\Lambda$. There exists a positive integer $h$ such that for $n\geq h$
$$\HH\HH^n(\Lambda) = \HH\HH^n(A)\oplus\HH\HH^n(B).$$

\end{theo}

\begin{proof}
According to the previous result, there exists $h$ such that $\MM^{\otimes_\AA h}=0$, and in this situation Corollary \ref{prodofcoho} provides the result.\qed
\end{proof}

\begin{exam}
Let $A=kQ_A$ and $B=kQ_B$ be path algebras, where $Q_A$ and $Q_B$ are finite quivers, possibly with oriented cycles. We view these quivers as situated  in horizontal plans, up and down. A new quiver $Q$ is obtained by adding  chosen sets of vertical up and down arrows which join vertices of $Q_A$ and $Q_B$.  Let $K$ be the two sided ideal of $kQ$ generated by the paths containing two vertical arrows, and let $\Lambda=kQ/K$.

Suppose  there are no efficient cycles with respect to the $(Q_A)_0\cup (Q_B)_0$-Peirce quiver of $\Lambda$. Equivalently, suppose there are no cycles in $Q$ of the form  $\delta v \dots  v \delta v \delta v$ or $v \dots  v \delta v \delta v$ which have at least one vertical arrow, and where the $v$'s belong to the set of vertical arrows and the $\delta$'s to the set of paths of $Q_A$ or of $Q_B$. Therefore $\HH\HH^n(\Lambda) = 0$ for $n$ large enough.

Indeed, by Theorem \ref{noefficientcycles} there exists a positive integer $h$ such that $\HH\HH^n(\Lambda) = \HH\HH^n(A)\oplus\HH\HH^n(B)$ if $n\geq h$. Moreover, Hochschild cohomology of a path algebra vanishes in degrees larger or equal to $2$, see e.g. \cite[p. 98]{CIBILS1991}.
\end{exam}

\section{Examples}\label{examples}

In this section we will exhibit some examples of algebras whose Hochschild cohomology can be described using the results obtained in this article. The first family of examples concerns toupie algebras. The second family aims to extend the answers already found in \cite{BUCHWGREENMADSENSOLBERG}   to Happel's remark in \cite{HAPPEL}.

\subsection{ Toupie algebras}

This family of triangular algebras has the following description by generators and relations.
The quiver $Q$ is as follows:
\medskip $$
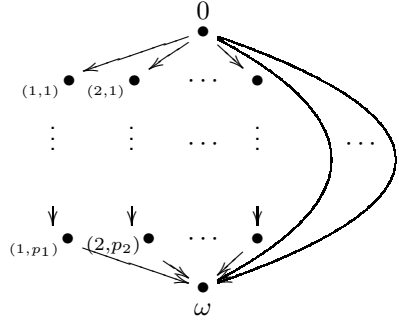
\xymatrix@R=6pt@C=6pt{  & & \ar[lld] \ar[ld]\bullet\save[]+<0pt,8pt>*{0}\restore  \ar[dr]\restore  \ar@/^4pc/[ddddd] \ar@/^6pc/[ddddd]&  \\
 \!\! _{_{(1,1)}} \bullet & \!\!\!\!\!\!\!\!_{_{(2,1)}}\bullet  & \cdots  & \bullet & \\ \vdots & \vdots &\cdots &\vdots& \,\,\,\,\,\,\,\,\,\,\,\cdots \\ \ar[d] & \ar[d] & & \ar[d]\\\!\!\!\!\!_{_{(1,p_1)}}\bullet\ar[drr] &  \!\!\!\!\!_{{(2,p_2)}}\bullet \ar[dr] & \cdots & \bullet\ar[ld] & \\& & \bullet\save[]+<0pt,-8pt>*{\omega}\restore & & }$$      \vspace{.1 cm}\\

There are two kinds of relations: monomial ones, and linear combinations of branches. The dimensions of the Hochschild cohomology spaces of toupie algebras - which are zero for $n>>0$ - have been computed in \cite{GaLa}, while explicit bases of these spaces have been described in \cite{Art}.

In case the quiver has a branch $\alpha$ containing at least three arrows and monomial relations involving neither the first arrow of the branch nor the last one, or, it has a branch with a monomial relation from the source vertex  to the sink, we shall see now that the algebra fits into our context. For this, we fix some notation.
Let $R$ denote a minimal set generating the ideal of relations of our algebra. Set $Q'=Q\setminus \alpha$, and $R'$ a minimal set generating the ideal of relations of
$\langle R \rangle \cap kQ'$. The quiver $Q''$ corresponds to the branch $\alpha$ without the extremal arrows $\alpha_0$ and $\alpha_n$.
Its arrows, counting from the source to the sink, are $\alpha_1, \dots, \alpha_{n-1}$.
Let $R''$ be a minimal set generating the ideal of relations of $\langle R \rangle\cap kQ''$. Fix $A=kQ'/\langle R' \rangle$ and $B=kQ''/\langle R'' \rangle$ and vertical arrows $\alpha_0$ from $A$ to $B$ and $\alpha_n$ from $B$ to $A$, so that the quiver obtained is the original one.
A study of the results of the previous section show  that for $n\geq 2$, the Hochschild cohomology of the toupie algebra is the direct sum of the Hochschild cohomology spaces of
$A$ and $B$.

\subsection{Happel's remark}

\begin{defi}
A \emph{DH-algebra} is a finite dimensional $k$-algebra of infinite global dimension with Hochschild  cohomology vector spaces zero in large enough degrees.
\end{defi}

D. Happel remarked in \cite{HAPPEL} that no DH-algebra was known.  Let $Q_A$ be the quiver with one vertex $s$ and two loops $a$ and $b$. It is shown in \cite{BUCHWGREENMADSENSOLBERG} that whenever  $q\in k$ is not a root of unity and $q\neq 0$, then $kQ_A/\langle a^2, b^2, ba-qab\rangle$ is a DH-algebra.

Let $Q_B = {}_x\cdot \stackrel{h}{\rightarrow} \cdot_y$. We add to the disjoint union of $Q_A$ and $Q_B$ two arrows: an ``up" one $u$ from $x$ to $s$, and a ``down" one $v$ from $s$ to $y$, obtaining a quiver $Q$.
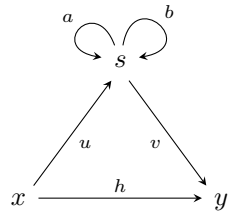
\[
\begin{tikzcd}[arrow style=tikz,>=stealth,row sep=4em]
 & s \arrow[out=110,in=170,loop,swap,"a"]
  \arrow[out=80,in=10,loop, "b"]
  \arrow[dr,shift right=.4ex,swap,"v"] &
 \\
x \arrow[ur,shift right=.4ex,swap,"u"] \arrow[rr,"h"] & & y
 &
\end{tikzcd}
\]
From the results in the previous section, it follows that for $q$ as above, the algebra
$$kQ/\langle a^2, b^2, ba-qab, vu, vau, vbu, vbau\rangle$$ is also a DH-algebra.

In order to construct other related DH-algebras, consider first $Q_A'$  the quiver obtained from $Q_A$ by adding an arrow $c$  from $s$ to a new vertex $t$. The algebra $kQ_A'/\langle a^2, b^2, ba-qab\rangle$ is still a DH-algebra. Let now $C$ be a finite quiver, $J$ an admissible two sided ideal of the path algebra $kC$, such that the Hochschild cohomology vector spaces of $kC/J$ are zero for large enough degrees. Suppose there are two vertices $x$ and $y$  of $C$ such that $x(kC/J)y=0$. Let $Q$ be the disjoint union of $Q_A'$ and $C$, together with an ``up" arrow from $x$ to $t$ and a ``down" arrow from $s$ to $y$. The results of the previous section show  that for $q\in k$ as above, $kQ/\langle a^2, b^2, ba-qab, J\rangle$ is a DH-algebra.

\[
\begin{tikzcd}[arrow style=tikz,>=stealth,row sep=4em]
t &\arrow[l,"c"] s \arrow[out=110,in=150,loop,swap,"a"]
  \arrow[out=50,in=10,loop, "b"]
  \arrow[d,swap,"v"]
 \\
x \arrow[u,swap,"u"]   & \arrow[l,dotted,"0"]y
\end{tikzcd}
\]

\footnotesize
\noindent C.C.:\\
Institut Montpelli\'{e}rain Alexander Grothendieck, CNRS, Univ. Montpellier, France.\\
{\tt Claude.Cibils@umontpellier.fr}

\medskip
\noindent M.L.:\\
Instituto de Matem\'atica y Estad\'\i stica  ``Rafael Laguardia'', Facultad de Ingenier\'\i a, Universidad de la Rep\'ublica, Uruguay.\\
{\tt marclan@fing.edu.uy}

\medskip
\noindent E.M.:\\
Departamento de Matem\'atica, IME-USP, Universidade de S\~ao Paulo, Brazil.\\
{\tt enmarcos@ime.usp.br}

\medskip
\noindent A.S.:
\\IMAS-CONICET y Departamento de Matem\'atica,
 Facultad de Ciencias Exactas y Naturales,\\
 Universidad de Buenos Aires, Argentina. \\{\tt asolotar@dm.uba.ar}

\end{document}